\documentclass[a4paper,11pt]{amsart}
\usepackage[T1]{fontenc}
\usepackage{amsmath,amsthm,amssymb,amsfonts}

\usepackage{wrapfig}

\usepackage{hyperref}
\usepackage[ansinew]{inputenc}
\usepackage{lmodern}

\usepackage{ifpdf}
 \ifpdf
  \usepackage[pdftex]{graphicx}
  \usepackage[protrusion=true,expansion=true]{microtype}
 \else
  \usepackage[dvips]{graphicx}
 \fi

\usepackage[margin=3cm]{geometry}

\newcommand{\romannum}{\renewcommand{\labelenumi}{\textnormal{(\roman{enumi})}}}
\newcommand{\cfct}{\mathbf{c}}

\newcommand{\C}{\mathbb{C}}

\newcommand{\N}{\mathbb{N}}

\newcommand{\R}{\mathbb{R}}

\newcommand{\Z}{\mathbb{Z}}

\newtheorem{thm}{Theorem}[section]

\newtheorem{lem}[thm]{Lemma}

\newtheorem{prop}[thm]{Proposition}

\theoremstyle{definition}
\newtheorem{defin}[thm]{Definition}
\newtheorem{rem}[thm]{Remark}
\newtheorem{exa}[thm]{Example}

\title[A disc multiplier in Jacobi Analysis]{Almost Everywhere Convergence of the inverse Jacobi Transform and Endpoint Results for a Disc Multiplier}
\author{Troels Roussau Johansen}
\address{Mathematisches Seminar,\\
   Christian-Albrechts Universit\"at zu Kiel\\
   Ludewig-Meyn-Strasse 4, D-24098 Kiel\\
   Germany}
\email{johansen@math.uni-kiel.de}
\date{}
\newcommand{\dfct}{\mathbf{d}}
\begin{document}
\begin{abstract}
The maximal operator $S_*$ for the spherical summation operator (or \emph{disc multiplier}) $S_R$ associated with the Jacobi transform through the defining relation $\widehat{S_Rf}(\lambda)=1_{\{\vert\lambda\vert\leq R\}}\widehat{f}(t)$ for a function $f$ on $\R$ is shown to be bounded from $L^p(\R_+,d\mu)$ into $L^p(\R,d\mu)+L^2(\R,d\mu)$ for $\frac{4\alpha+4}{2\alpha+3}<p\leq 2$. Moreover $S_*$ is bounded from $L^{p_0,1}(\R_+,d\mu)$ into $L^{p_0,\infty}(\R,d\mu)+L^2(\R,d\mu)$. In particular $\{S_Rf(t)\}_{R>0}$ converges almost everywhere towards $f$, for $f\in L^p(\R_+,d\mu)$, whenever $\frac{4\alpha+4}{2\alpha+3}<p\leq 2$.
\end{abstract}


\maketitle

\section{Introduction}
\label{intro}
The importance of the disc multiplier in Euclidean harmonic analysis -- defined as the operator $S_R$ satisfying the relation $\widehat{S_Rf}(\xi)=1_{\|\xi\|\leq R}\widehat{f}(\xi)$ -- was firmly established by Fefferman's groundbraking result in \cite{Fefferman-ballmultiplier} that $S_R$ is not bounded on $L^2(\R^n)$, $n\geq 2$, unless $p=2$. The operator has since then played a role in other areas of mathematics. It usually appears whenever one studies convergence properties of eigenfunction expansions for differential operators on manifolds, and it also appears as an extreme endpoint case of Bochner--Riesz means. An interesting aspect, however, is that the operator behaves much better when restricted to radial $L^p$-functions. Indeed, according to \cite{Herz}, the operator is bounded on $L^p_\text{rad}(\R^n)$ for $\frac{2n}{n+1}<p<\frac{2n}{n-1}$. This result has later been improved in several directions, and we shall recall them one by one in the main text.

A natural analogue of the disc multiplier in the framework of spherical analysis on Riemannian symmetric spaces of rank one was introduced by Meaney and Prestini in the mid-90's and the study was completed in the paper \cite{Meaney-Prestini_invtrans} with almost sharp statements about the mapping properties of the maximal operator associated with the disc multiplier. In the present paper we follow in their footsteps and generalize their results to Jacobi analysis, and we establish the missing endpoint results in the setting of Jacobi analysis. In particular we complement the paper \cite{Anker-Damek-Yacoub}. This implies  almost everywhere convergence of $\{S_Rf(x)\}_{R>0}$ for $f\in L^p(d\mu)$ for a certain range of $p$, most directly related to \cite{RomeraSoria} in the Euclidean case, whereas the extension to Hankel transforms was considered in \cite{Colzani-Hankel}.

There are other ways to obtain  almost everywhere convergence of $\{S_Rf(x)\}_{R>1}$. In \cite{Brandolini-Gigante}, the authors obtain equiconvergence results for $\{S_Rf\}$ in the slightly more general framework of (noncompact) {C}h\'ebli--{T}rim\`eche hypergroups. The results of the present paper should generalize to their setting without much effort. Our endpoint results are stronger, however, as we are able to determine the endpoint behavior of the maximal operator at the level of Lorentz spaces. Moreover, and this is a fundamental advantage of working with maximal operators, we will use the results of the present paper as part of a complex interpolation argument in a companion paper to obtain convergence results for Bochner--Riesz means in Jacobi analysis below the critical order of integrability. In order for this to work we need norm estimates in the first place.

Finally we wish to point out that a ``flat'' version of our results on the disc multipliers were recently obtained in \cite{ElKamel-Yacoub}. By ``flat'' we refer to the modern habit of regarding Dunkl analysis on $\R$ as a `zero curvature limit' of harmonic analysis in rank one root systems, in the sense of Cherednik, Heckman and Opdam. The proofs of \cite{ElKamel-Yacoub} are more or less straightforward adaptations of techniques from \cite{RomeraSoria} and \cite{Prestini-summation}, since the size of balls, measured in terms of the relevant measures in Dunkl theory, do not grow exponentially fast, contrary to what happens for the noncompact Riemannian symmetric spaces. It is well-understood that the ``curved'' situation -- be it analysis on symmetric spaces or slightly more generally in Jacobi analysis -- is complicated by balls having exponential volume growth.
\medskip

We employ the same techniques as in \cite{Meaney-Prestini_invtrans}, carried out in the more general setting of Jacobi analysis. Most proofs are therefore structurally identical to those in \cite{Meaney-Prestini_invtrans}, which we wish to acknowledge at this point. There are several technical difficulties, however, like the precise asymptotic expansion for the $\cfct$-function in Lemma \ref{lemma.precise-c}. Also of importance was that we were able to incorporate the paper \cite{Prestini-weighted} by Prestini. The careful analysis, in turn, allowed us to establish new endpoint results, thereby showing to exactly what extend one can generalize the spherical analysis on symmetric spaces of rank one. Since we never use the actual formula for the measures $d\mu(t)$, but rather just its behavior for $t\thicksim 1$ and $t\gg 1$, and since the key ingredients for the proofs -- asymptotic estimates for $\varphi_\lambda$ and the Plancherel density $\vert\cfct(\lambda)\vert^{-2}$ -- are also available for {C}h\'ebli--{T}rim\`eche hypergroups (see Theorem~1.2, Section~1.3, Theorems~2.1 and 2.2 in \cite{Brandolini-Gigante}) the exact same calculations can be carried out in the context of such hypergroups.

\section{Jacobi Analysis}
Let $(a)_0=1$ and $(a)_k=a(a+1)\cdots (a+k-1)$. The hypergeometric function ${_2}F_1(a,b;c,z)$ is defined by
\[{_2}F_1(a,b;c,z)=\sum_{k=0}^\infty\frac{(a)_k(b)_k}{(c)_kk!}z^k,\quad\vert z\vert<1;\]
the function $z\mapsto{_2}F_1(a,b;c,z)$ is the unique solution of the differential equation
\[z(1-z)u''(z)+(c-(a+b+1)z)u'(z)-abu(z)=0\]
which is regular in $0$ and equals $1$ there.  The Jacobi functions for parameters $(\alpha,\beta)$ are defined by $\varphi_\lambda^{(\alpha,\beta)}(t)={_2}F_1(\frac{1}{2}(\alpha+\beta+1-i\lambda),
\frac{1}{2}(\alpha+\beta+1+i\lambda); \alpha+1,-\sinh^2t)$. It is
thereby clear that $\lambda\mapsto\varphi_\lambda(t)$ is analytic for all
$t\geq 0$. Moreover, for $\Im\lambda\geq 0$, there exists a unique
solution $\phi_\lambda$ to the same equation satisfying
$\phi_\lambda(t)=e^{(i\lambda-\rho)t}(1+o(1))$ as $t\to\infty$, and
$\lambda\mapsto\phi_\lambda(t)$ is therefore also analytic for $t\geq 0$.

In what follows we assume that $\alpha\neq -1,-2,\ldots$, $\alpha>\beta>-\frac{1}{2}$, and $\vert\beta\vert<\alpha+1$. Let $\rho=\alpha+\beta+1$. The usual Lebesgue space on $\R^+$ shall simply be denoted $L^p$, whereas by $L^p(d\mu)$ we understand the weighted Lebesgue space, with $d\mu(t)=d\mu_{\alpha,\beta}(t)=\Delta(t)\,dt$, where $\Delta_{\alpha,\beta}(t)=(2\sinh t)^{2\alpha+1}(2\cosh t)^{2\beta+1}$, $t>0$. We adopt the notational convention of writing $\mu(A)$ for the weighted measure of a measurable subset $A$ of $\R$, that is, $\mu(A)=\|1_A\|_{L^1(d\mu)}$. It is of paramount importance to stress that the behavior of $\Delta(t)$ depends on the `size' of $t$. More precisely,

\[\begin{split}
\vert\Delta_{\alpha,\beta}(t)\vert \leq
\begin{cases} t^{2\alpha+1}&\text{for }
t\lesssim 1\\
e^{2\rho t}&\text{ for } t\gg 1\end{cases}.
\end{split}\]

In analogy with the case of symmetric spaces, one proceeds to show the existence of a function
$\cfct=\cfct_{\alpha,\beta}$ for which
$\varphi_\lambda(t)=\cfct(\lambda)e^{(i\lambda-\rho)t}\phi_\lambda(t)+\cfct(-\lambda)e^{(-i\lambda-\rho)t}\phi_{-\lambda}(t)$.
Since we adhere to the conventions and normalization used in
\cite{Koornwinder-FJ}, the $\cfct$-function is given by
\[\cfct(\lambda)=\frac{2^\rho\Gamma(i\lambda)\Gamma(\frac{1}{2}(1+i\lambda))}
{\Gamma(\frac{1}{2}(\rho+i\lambda))\Gamma(\frac{1}{2}(\rho+i\lambda)-\beta)}.\]
Observe that for $\alpha,\beta\neq -1,-2,\ldots$, $\cfct(-\lambda)^{-1}$ has
finitely many poles for $\Im\lambda <0$ and none if $\Im\lambda\geq 0$ and $\Re\rho>0$. It follows from Stirling's formula that
for every $r>0$ there exists a positive constant $c_r$ such  that
\begin{equation}\label{eqn.c-fct.est}
\vert\cfct(-\lambda)\vert^{-1}\leq c_r(1+\vert\lambda\vert)^{\alpha+\frac{1}{2}}\text{ if } \Im\lambda\geq 0\text{ and }
\cfct(-\lambda')\neq 0\text{ for }\vert\lambda'-\lambda\vert\leq r.
\end{equation}
The following statement on the precise asymptotic expansion of the density $\vert\cfct(\lambda)\vert^{-2}$ will play an important role later in the paper. We have included a detailed proof since the result cannot be lifted directly from \cite{Stanton-Tomas} or \cite{Meaney-Prestini_invtrans}; $\alpha$ and $\beta$ need not correspond to integer-valued root multiplicities, so the expression for $\cfct(\lambda)$ does not really simplify, unlike for rank one symmetric spaces.

\begin{lem}\label{lemma.precise-c} Assume $\alpha>\beta>-\frac{1}{2}$.
 \begin{enumerate}
 \romannum
 \item For every integer $M$ there exist constants $c_i, i=0,\ldots,M-1$ (depending on $\alpha$, $\beta$, and $M$) such that
     \[\vert\cfct(\lambda)\vert^{-2}\thicksim c_0\vert\lambda\vert^{2\alpha+1}\Biggl\{1+\sum_{j=1}^{M-1}c_j\lambda^{-j}+O\bigl(\lambda^{-M}\bigl)\Biggr\}\text{ as } \vert\lambda\vert\to\infty.\]
\item Let $\dfct(\lambda)=\vert\cfct(\lambda)\vert^{-2}$,
$\lambda\geq 0$, and $k\in\N_0$. There exists a constant
$c_k=c_{k,\alpha,\beta}$ such that
\[\Bigl|\frac{d^k}{d\lambda^k}\dfct(\lambda)\Bigr|\leq
c_k(1+\vert\lambda\vert)^{2\alpha+1-k}.\]
\item $\cfct'(\lambda)\thicksim \cfct(\lambda)O(\lambda^{-1})$ and $\cfct''(\lambda)\thicksim\cfct(\lambda)O(\lambda^{-2})$.
\end{enumerate}
\end{lem}
This improves on the usual asymptotic statement that $\vert\cfct(\lambda)\vert^{-2}\thicksim \vert\lambda\vert^{2\alpha+1}$ as $\vert\lambda\vert\to\infty$ and we will need this improvement at a later stage. This was already observed in \cite{Meaney-Prestini_invtrans}.

\begin{proof}
\noindent (i)\indent Following the technique in \cite[Subsection~2.2.1]{Paris-Kaminski} we introduce the auxiliary function
\[Q(\lambda)=\Biggl(\prod_{r=1}^q\Gamma(1-b_r+\beta_r\lambda)\Biggr)\Biggl(\prod_{r=1}^p\Gamma(1-a_r+\alpha_r\lambda)\Biggr)^{-1},\]
where we of course have in mind the particular parameters
\begin{equation}\label{eqn.parameters}
\left\{\begin{aligned}
p&=2, q=4, &b_1&=b_2=1-\frac{\rho}{2}, &b_3=b_4&=1+\beta-\frac{\rho}{2}\\
a_1&=a_2=1, &\beta_1&=\beta_2=\beta_3=\beta_4=\frac{i}{2},  &\text{and } &\alpha_1=\alpha_2=i
\end{aligned}\right.
\end{equation}
so that $\vert Q(\lambda)\vert=\vert\cfct(\lambda)\vert^{-2}$ by the duplication formula for the $\Gamma$-function. Recall that by Stirling's formula,
\[\log\Gamma(z)=(z-\tfrac{1}{2})\log z-z+\tfrac{1}{2}\log(2\pi)+\Omega(z), \text{ where } \Omega(z)\thicksim\sum_{r=1}^\infty\frac{B_{2r}}{2r(2r-1)z^{2r-1}}\]
for suitable numbers $B_{2n}$ (the Bernoulli numbers). Moreover,
\[\Omega(z)=\sum_{r=1}^{n-1}\frac{B_{2r}}{2r(2r-1)z^{2r-1}}+R_n(z)\]
for every positive integer $n$, where -- upon writing $z=xe^{i\theta}$ -- the remainder term $R_n(z)$ may be estimated according to
\begin{equation}\label{eqn.remainder-estimate}
\vert R_n(z)\vert\leq \frac{\vert B_{2n}\vert}{2n(2n-1)}\frac{(\sec\tfrac{\theta}{2})^{2n}}{\vert z\vert^{2n-1}}\text{ for }\vert\arg z\vert<\pi,
\end{equation}
see \cite[Equation~2.1.6]{Paris-Kaminski}. Presently $z$ will be of the form $z=\alpha_r\lambda+1-a_r$ with $\alpha_r\geq 0, a_r\in\C$, and $\lambda\in\R_+$, so that $\arg z$ remains constant as $\lambda\to\infty$.

For fixed $M\in\N$ and large $\vert\lambda\vert$ it thus holds that
\[\begin{split}
\log Q(\lambda)&=\sum_{r=1}^q\log\Gamma(1-b_r+\beta_r\lambda)-\sum_{r=1}^p\log\Gamma(1-a_r+\alpha_r\lambda)\\
&=\sum_{r=1}^q\Bigl\{\bigl(\tfrac{1}{2}-b_r+\beta_r\lambda\bigr)\log(1-b_r+\beta_r\lambda)-(1-b_r+\beta_r\lambda) \\
&\qquad +\tfrac{1}{2}\log(2\pi)+\Omega(1-b_r+\beta_r\lambda)\Bigr\}\\
&\quad - \sum_{r=1}^p \Bigl\{\bigl(\tfrac{1}{2}-a_r+\alpha_r\lambda\bigr)\log(1-a_r+\alpha_r\lambda)-(1-a_r+\alpha_r\lambda)\\ &\qquad+\tfrac{1}{2}\log(2\pi)+\Omega(1-a_r+\alpha_r\lambda)\Bigr\}\\
&\lessapprox \sum_{r=1}^q\bigl(\tfrac{1}{2}-b_r+\beta_r\lambda\bigr)\log(\beta_r\lambda)-\sum_{r=1}^p\bigl(\tfrac{1}{2}-a_r+\alpha_r\lambda\bigr) \log(\alpha_r\lambda)\\
&\qquad +\frac{1}{2}(q-p)(\log(2\pi)-2)-s\kappa-\theta\\
&\qquad + \sum_{r=1}^q\Omega(1-b_r+\beta_r\lambda)-\sum_{r=1}^p\Omega(1-a_r+\alpha_r\lambda)\\
&\thicksim
\sum_{r=1}^q\bigl(\tfrac{1}{2}-b_r+\beta_r\lambda\bigr)(\log \beta_r+\log\lambda)-\sum_{r=1}^p\bigl(\tfrac{1}{2}-a_r+\alpha_r\lambda\bigr) (\log \alpha_r+\log\lambda)\\
&\quad+\frac{1}{2}(q-p)(\log(2\pi)-2)-s\kappa-\theta+\sum_{j=1}^{M-1}c_j\lambda^{-j}+O(\lambda^{-M})\\
&=\Bigl[\frac{q-p}{2}+\theta+\lambda\kappa\Bigr]\log\lambda - \lambda(\log h+\kappa)\\ &\quad+ \log \widetilde{c}_0-\theta+\frac{q-p}{2}(\log(2\pi)-2)
+\sum_{j=1}^{M-1}c_j\lambda^{-j}+O(\lambda^{-M})
\end{split}\]
where
\[h:=\Bigl(\prod_{r=1}^p\alpha_r^{\alpha_r}\Bigr)\Bigl(\prod_{r=1}^q\beta_r^{-\beta_r}\Bigr),\quad \widetilde{c}_0:=\Bigl(\prod_{r=1}^q\beta_r^{\frac{1}{2}-b_r}\Bigr)\Bigl(\prod_{r=1}^p\alpha_r^{a_r-\frac{1}{2}}\Bigr)\]
\[\theta:=\sum_{r=1}^pa_r-\sum_{r=1}^qb_r,\text{ and } \kappa:=\sum_{r=1}^q\beta_r-\sum_{r=1}^q\alpha_r.\]
\noindent With the parameters defined as in \eqref{eqn.parameters}, one sees that $\kappa=0$, $\theta=1+1-(1-\tfrac{\rho}{2})\times 2-(1-\tfrac{\rho}{2}+\beta)\times 2 = 2\alpha$, and $\frac{q-p}{2}=1$, whence
\[\begin{split}
Q(\lambda)&\thicksim c_0\lambda^{\frac{q-p}{2}+\theta+s\kappa}e^{-s\kappa}\Biggl\{1+\sum_{j=1}^{M-1}c_j\lambda^{-j}+O(\lambda^{-M})\Biggr\}\\
&=c_0\lambda^{2\alpha+1}\Biggl\{1+\sum_{j=1}^{M-1}c_j\lambda^{-j}+O(\lambda^{-M})\Biggr\}\text{ as } \vert\lambda\vert\to\infty.
\end{split}\]
\bigskip

\noindent (ii) and (iii)\indent Moreover,
$\dfct'(\lambda)=-2\dfct(\lambda)\frac{\cfct'(\lambda)}{\cfct(\lambda)}$, so it
suffices to show that $\frac{\cfct'(\lambda)}{\cfct(\lambda)}\simeq
O(\frac{1}{\lambda})$. This may be seen as in the proof of
\cite[Lemma~8]{Meaney-Prestini} as follows: Since
\[\frac{\cfct'(\lambda)}{\cfct(\lambda)}=ic_{\alpha,\beta}\Bigl\{\psi(i\lambda)-\psi(\alpha-\beta+i\lambda)+\frac{1}{2}\psi(\frac{\alpha-\beta+i\lambda}{2})
-\frac{1}{2}\psi(\frac{\rho+i\lambda}{2})\Bigr\},\] where
$\psi(z):=\frac{\Gamma'(z)}{\Gamma(z)}=-\gamma-\frac{1}{z}+\sum_{n=1}^\infty\frac{z}{n(n+z)}$,
$\gamma$ being the Euler constant, it follows that
\[\frac{\cfct'(\lambda)}{\cfct(\lambda)}=c_{\alpha,\beta}\Bigl\{-\frac{1}{i\lambda}+\frac{1}{\rho+i\lambda}+\sum_{n=1}^\infty\frac{z_1-z_3}{(z_1+n)(z_3+n)}
+ \frac{1}{2}\sum_{n=1}^\infty\frac{z_2-z_4}{(z_2+n)(z_4+n)}\Bigr\}\] with
$z_1=i\lambda$, $z_2=\frac{1}{2}(\alpha-\beta+i\lambda)$,
$z_3=\alpha-\beta+i\lambda$, and $z_4=\frac{1}{2}(\rho+i\lambda)$. Observe that
$\left|-\frac{1}{i\lambda}+\frac{1}{\rho+i\lambda}\right|=\left|\frac{-\rho}{i\lambda(\rho+i\lambda)}\right|\leq\frac{1}{\vert\lambda\vert}\frac{\vert\rho\vert^2+\vert\rho\vert\vert\lambda\vert}{\vert\rho\vert^2
+\vert\lambda\vert^2}\leq \frac{c}{\vert\lambda\vert}$. The assertion for $k=1$
now follows from the estimate
\begin{multline*}
\frac{\vert\alpha-\beta\vert}{2}\sum_{n=1}^\infty\frac{1}{\vert z_1+n\vert\vert
z_3+n\vert}+\frac{\vert
\beta+\frac{1}{2}\vert}{2}\sum_{n=1}^\infty\frac{1}{\vert z_2+n\vert \vert
z_4+n\vert} \\ \leq c_{\alpha,\beta}\int_1^\infty
\frac{1}{x^2+\lambda^2}\,dx\leq
\frac{c_{\alpha,\beta}}{\vert\lambda\vert}.\end{multline*}
\medskip

The required estimates for $\dfct''(\lambda)$ is obtained analogously: first
observe that
$\dfct''(\lambda)=-2\dfct'(\lambda)\frac{\cfct'(\lambda)}{\cfct(\lambda)}-2\dfct(\lambda)\bigl((\frac{\cfct''(\lambda)}{\cfct(\lambda)}+\bigl(\frac{\cfct'(\lambda)}{\cfct(\lambda)}\bigr)^2
\bigr)$. In order to establish the assertion in the Proposition for the case
$k=2$ it suffices to prove that $\frac{\cfct''(\lambda)}{\cfct(\lambda)}\simeq
O(\frac{1}{\lambda^2})$. This can also be established as in the proof of
\cite[Lemma~8]{Meaney-Prestini}; indeed,
\begin{multline*}
\frac{\cfct''(\lambda)}{\cfct(\lambda)}=c_{\alpha,\beta}(\psi'(z_1)-\psi'(z_2)+\tfrac{1}{4}\psi'(z_3)-\tfrac{1}{4}\psi'(z_4))\\
+c_{\alpha,\beta}\frac{\cfct'(\lambda)}{\cfct(\lambda)}(\psi(z_1)-\psi(z_2)+\tfrac{1}{2}\psi(z_3)-\tfrac{1}{2}\psi(z_4))\end{multline*}where
$\psi'(z)=z^{-2}+\sum_{n=1}^\infty(z+n)^{-2}$, evaluated in the four points
$z_i$. Heuristically, it is now easy to prove that the left hand side is
$O(\lambda^{-2})$. The second half of the right hand side behaves as
$\frac{1}{\lambda}\frac{1}{\lambda}$ according to what we have already
established in the case of $k=1$, so it suffices to show that
\[\biggl|\sum_{n=1}^\infty\frac{1}{(z_i+n)^2}-\frac{1}{(z_{i+2}+n)^2}\biggr|\leq\frac{c}{\vert\lambda\vert^2}\text{
for } i=1,2.\] If, say, $i=1$, the required estimate follows like this:
\[\begin{split}
\biggl|\sum_{n=1}^\infty\frac{1}{(z_1+n)^2}-\frac{1}{(z_3+n)^2}\biggr|&\leq
\sum_{n=1}^\infty\biggl|\frac{(\alpha-\beta)(\alpha-\beta+2i\lambda+2n)}{(i\lambda+n)^2(\alpha-\beta+i\lambda+n)^2}\biggr|\\
&\leq c_{\alpha,\beta}\int_1^\infty\frac{1}{x^3+\lambda^3}\,dx\leq \frac{c_{\alpha,\beta}'}{\vert\lambda\vert^2}.
\end{split}\]

One proves by induction that
$\frac{\cfct^{(k)}(\lambda)}{\cfct(\lambda)}=O(\lambda^{-k})$ for
$k=0,1,\ldots$, and one would then formally have to carry out another proof by
induction that the estimate for $\dfct^{(k)}(\lambda)$ have the right order in
$\vert\lambda\vert$. We leave the  tedious details to the energetic reader.
\end{proof}
\begin{rem}
In principle one should be able to obtain the asymptotic expansion for $\vert\cfct(\lambda)\vert^{-2}$ from the expansion of $\vert\cfct(\lambda)\vert^2$ in \cite[Theorem~2.2]{Brandolini-Gigante} by long division of the asymptotic series. Such computations are indeed justified, cf. \cite[Section~1.5]{Erdelyi-asymptotic}. We have opted for a self-contained proof, however, that is inspired by \cite[Section~2.2]{Paris-Kaminski}. We found it worthwhile to use the explicit formula for the $\cfct$-function since we still have to have similar estimates for various derivatives of $\cfct(\lambda)$ and $\vert\cfct(\lambda)\vert^{\pm 2}$. The asymptotic expansion for $\vert\cfct(\lambda)\vert^{-2}$ will therefore be more explicit than what could obtained from \cite{Brandolini-Gigante}.
\end{rem}

\begin{exa}[Specialization to rank one symmetric spaces]
For special values of $\alpha$ and $\beta$, determined by the root system of a
rank one Riemannian symmetric space, the functions $\varphi_\lambda$ are the
usual spherical functions of Harish-Chandra, and the Jacobi transform is the spherical transform. To be more precise
assume $G/K$ is a rank one Riemannian symmetric space of noncompact type, with
positive roots $\alpha$ and $2\alpha$. Furthermore let $p$ denote the
multiplicity of $\alpha$ and $q$ the multiplicity of $2\alpha$ (we allow $q$ to
be zero). With $\alpha:=\frac{1}{2}(p+q-1)$ and $\beta:=\frac{1}{2}(q-1)$ both real, and $p=2(\alpha-\beta)$ and $q=2\beta+1$, the
function $\varphi^{(\alpha,\beta)}_\lambda$ is precisely the usual elementary
spherical function $\varphi_\lambda$ as considered by Harish-Chandra, $\rho=\alpha+\beta+1=\frac{1}{2}(p+2q)$ as it should be, and $\text{dim}(G/K)=p+q+1=2(\alpha+1)$. According to Lemma \ref{lemma.precise-c} we write $\vert\cfct(\lambda)\vert^{-2}=P(\lambda)+E(\lambda)$, where
\[\vert E(\lambda)\vert = \vert P(\lambda)\vert\cdot\begin{cases} 0&\text{whenever } q=0, p=2k\\
\bigl|1-\coth\bigl(\frac{\pi\lambda}{2}\bigr)\bigr|&\text{whenever } q=2l+1, p=4k+2\\ \bigl|1-\tanh\bigl(\frac{\pi\lambda}{2}\bigr)\bigr|&\text{otherwise}\end{cases}\]
cf. the proof of Lemma~4.2 in \cite{Stanton-Tomas}. Since $\vert\cfct(\lambda)\vert^{-2}\thicksim\lambda^{n-1}$ as $\lambda\to\infty$, we can at least say that $\text{deg }P(\lambda)=n-1$.

A similar choice of parameters $\alpha,\beta$ reveals that even spherical
analysis on Damek--Ricci spaces is subsumed by the present setup. This was already exploited in \cite{Anker-Damek-Yacoub} in order to extend results from spherical analysis on rank one symmetric spaces to the framework of Damek--Ricci spaces.
\end{exa}

Let $d\nu(\lambda)=d\nu_{\alpha,\beta}(\lambda)=(2\pi)^{-\frac{1}{2}}\vert\cfct(\lambda)\vert^{-2}\,d\lambda$ and denote by $L^p(d\nu)$ the associated weighted Lebesgue space on $\R^+$; note that $\cfct(\lambda)\cfct(-\lambda)=\cfct(\lambda)\cfct(\lambda)=\vert\cfct(\lambda)\vert^2$ whenever $\alpha,\beta,\lambda\in\R$. The Jacobi transform, initially defined for, say
a function $f\in C_c^\infty(\R^+)$ by
\[\widehat{f}(\lambda)=\frac{\sqrt{\pi}}{\Gamma(\alpha+1)}\int_0^\infty f(t)\varphi_\lambda(t)\,d\mu(t)\]
extends to a unitary isomorphism from $L^2(d\mu)$ onto $L^2(d\nu)$,
and the inversion formula is the statement that
\[f(t)=\int_0^\infty\widehat{f}(\lambda)\varphi_\lambda(t)\,d\nu(\lambda)\]
holds in the $L^2$-sense, cf. \cite[Formula~4.5]{Koornwinder-newproof}. The limiting case $\alpha=\beta=-\frac{1}{2}$ is the
Fourier-cosine transform, which we will not study. One easily verifies that
$\widehat{\mathcal{L}f}(\lambda)=-(\lambda^2+\rho^2)\widehat{f}(\lambda)$.
\section{The Disc Multiplier: Statement of Results}\label{sec.disc-results}
Our starting point in defining the disc multiplier is the inversion formula for the Jacobi transform, that is,
\[f(t)=\int_0^\infty \widehat{f}(\lambda)\varphi_\lambda(t)\,d\nu(\lambda),\]
where $d\nu(\lambda)=\vert\cfct(\lambda)\vert^{-2}\,d\lambda$.  Let  $S_Rf(t)=\int_0^R\widehat{f}(\lambda)\varphi_\lambda(t)\,d\nu(\lambda)$ and notice that for well-behaved functions $f$ (say, in $C_c^\infty(\R^+)$), $S_Rf$ may be written as an integral operator
\[
S_Rf(t)=\int_0^R\biggl\{\int_0^\infty f(r)\varphi_\lambda(r)\,d\mu(r)\biggr\}\varphi_\lambda(t)\,d\nu(\lambda)
=\int_0^\infty K_R(t,r)f(r)\,d\mu(r)\]
where $K_R(t,r)=\int_0^R\varphi_\lambda(t)\varphi_\lambda(r)\,d\nu(\lambda)$.
The goal of the present paper is to investigate the mapping properties of the associated maximal operator $S_*:f\mapsto S_* f$,
\[S_* f(t)=\sup_{R>0}\vert S_Rf(t)\vert\]
in order to establish almost everywhere convergence, $S_Rf(t)\to f(t)$, for $f$ in $L^p(d\mu)$, for a nontrivial range of $p$. The investigation follows  \cite{Meaney-Prestini_invtrans} very closely, but several complications of a purely technical nature (the Jacobi parameters $\alpha$, $\beta$ not being integers, for example), will make the presentation lengthier. The philosophy is simple, however; since the functions $\varphi_\lambda$ behave locally as an Euclidean eigenfunction (meaning a Bessel function since we always have the spherical analysis in mind), we should analyse the kernel $K_R$ in different regions of the $(t,r)$-domain $\R_+\times\R_+$ to probe similarities with as well as deviations from a purely Euclidean harmonic analysis. This will imply a decomposition of $S_Rf$ as the sum $S_Rf(t)=\sum_{i=1}^4S_{i,R}f(t)$, where $S_{i,R}f(t)=\int_0^\infty K_{i,R}(t,r)f(r)\,d\mu(r)$ and $K_{i,R}(t,r)=1_{A_i}(t,r)K_R(t,r)$, $i=1,\ldots,4$, where

\begin{minipage}{\textwidth}
\vspace{1cm}
\begin{minipage}{0.4\textwidth}
\[\begin{split}
A_1&=\{(t,r)\,:\, 0\leq t, r\leq R_0\}\\
A_2&=\{(t,r)\,:\,t,r\gg R_0\}\\
A_3&=\{(t,r)\,:\,t\gg 1, 0\leq r\leq R_0\}\\	
A_4&=\{(t,r)\,:\, 0\leq t\leq R_0, r\gg R_0\}
\end{split}\]
\end{minipage}
\hspace{1cm}
\begin{minipage}{0.4\textwidth}
\includegraphics[scale=0.5]{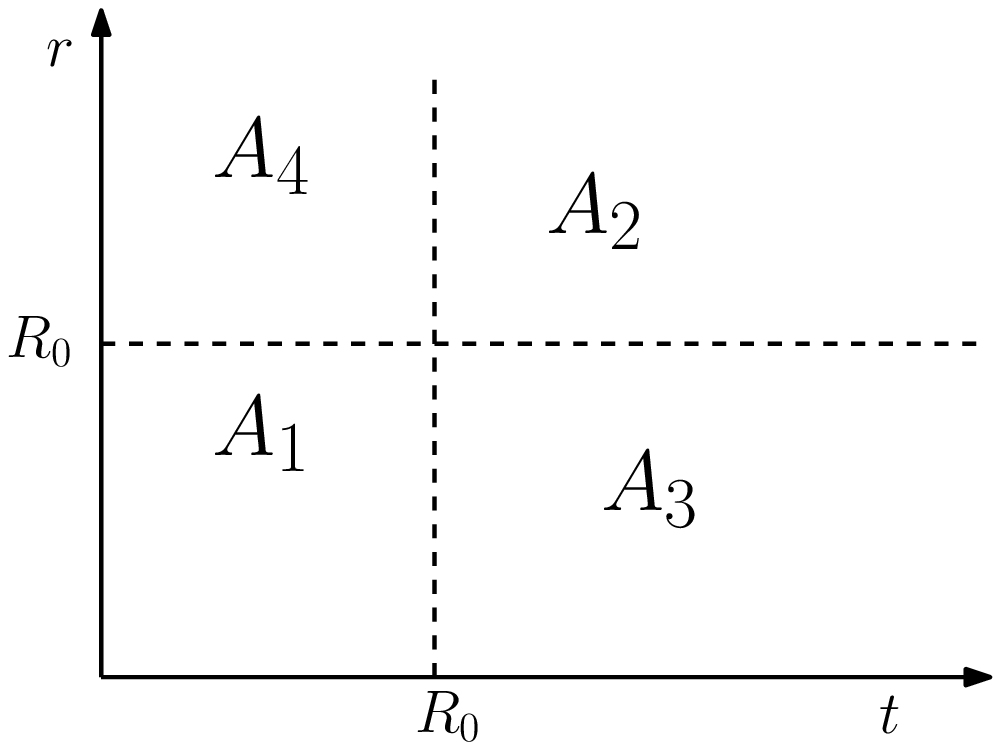}
\end{minipage}
\vspace{1cm}
\end{minipage}

To be more precise, the constant $R_0$ will be chosen as in the following technical lemma (the proof of which can be found in \cite{Stanton-Tomas} for rank one symmetric spaces and more generally for Jacobi functions in \cite{Johansen-exp1}). Here  $J_\mu(z)$ is the usual Bessel function of order $\mu$ and
$\mathcal{J}_\mu(z)$ is the modified Bessel function defined by
$\mathcal{J}_\mu(z)=2^{\mu-1}\Gamma(\frac{1}{2})\Gamma(\mu+\frac{1}{2})z^{-\mu}J_\mu(z)$.

\begin{lem}\label{lemma.asymptoticEXP}
Assume $\Re\alpha>\frac{1}{2}$, $\Re\alpha >\Re\beta>-\frac{1}{2}$, and that $\lambda$ belongs either to a compact subset of
$\C\setminus (-i\N)$ or a set of the form
\[D_{\varepsilon,\gamma}=\{\lambda\in\C\,:\, \gamma\geq\Im\lambda\geq -\varepsilon\vert\Re\lambda\vert\}\]
for some $\varepsilon,\gamma\geq 0$. There exist constants $R_0, R_1\in
(1,\sqrt{\frac{\pi}{2}})$ with $R_0^2<R_1$ such that for every $M\in\N$ and
every $t\in[0,R_0]$
\begin{eqnarray*}
\varphi_\lambda^{(\alpha,\beta)}(t) =
\frac{2\Gamma(\alpha+1)}{\Gamma(\alpha+\tfrac{1}{2})\Gamma(\tfrac{1}{2})}\frac{t^{\alpha+\frac{1}{2}}}{\sqrt{\Delta(t)}} \sum_{m=0}^\infty
a_m(t)t^{2m}\mathcal{J}_{m+\alpha}(\lambda t)
\\
\varphi_\lambda^{(\alpha,\beta)}(t) =
\frac{2\Gamma(\alpha+1)}{\Gamma(\alpha+\tfrac{1}{2})\Gamma(\tfrac{1}{2})}\frac{t^{\alpha+\frac{1}{2}}}{\sqrt{\Delta(t)}} \sum_{m=0}^M
a_m(t)t^{2m}\mathcal{J}_{m+\alpha}(\lambda t)+E_{M+1}(\lambda t),
\end{eqnarray*}
where
$a_0(t)\equiv 1$  and $\vert a_m(t)\vert\leq c_\alpha(t)R_1^{-(\Re\alpha+m-\frac{1}{2})}\text{ for all } m\in\N$. Additionally, the error term $E_{M+1}$ is bounded as follows:
\begin{equation*}
\vert E_{M+1}(\lambda t)\vert \leq \begin{cases} c_Mt^{2(M+1)}&\text{if }\vert\lambda t\vert\leq 1\\
c_M t^{2(M+1)}\vert\lambda t\vert^{-(\Re\alpha+M+1)}&\text{if
}\vert\lambda t\vert>1.\end{cases}
\end{equation*}
\end{lem}
\medskip

In the following four subsections will establish the mapping properties of the associated four maximal operators $S_{i*}$ individually, and the main theorem will then follow by noting that $\vert S_* f(t)\vert\leq\sum_{i=1}^4\vert S_{i,*}f(t)\vert$. The investigation in \cite{Meaney-Prestini_invtrans} and its outcome may be summarized roughly as follows:  For $f\in L^s(K\backslash G/K)$, we split the maximal operator $S_*$ associated to the ``disc multiplier'' as
\[S_*f=S_{1,*}f + S_{2,*}f + S_{3,*}f + S_{4,*}f,\]
where
\begin{itemize}
\item $S_{1,*}f$ is bounded on $L^s(K\backslash G/K)$ for $\frac{2n}{n+1}<s<\frac{2n}{n-1}$ (this is the `Herz range')
\item $S_{2,*}f$ is bounded into $L^2(G)+L^s(G)$ for $1<s\leq 2$
\item $S_{3,*}f$ is bounded into $L^2(G)$ for $1<s\leq 2$
\item $S_{4,*}f$ is bounded into $L^2(G)$ for $\frac{2n}{n+1}<s$.
\end{itemize}
It thus follows -- and this is the main result of the paper \cite{Meaney-Prestini_invtrans} - that $S_*f$ belongs to $L^2(G)+L^s(G)$ for $\frac{2n}{n+1}<s\leq 2$ (since $2\leq\frac{2n}{n-1}$ for all $n\in\N$). Our first result is a generalization theoreof to the setting of Jacobi analysis. For the remainder of the paper we set $p_0=\frac{4\alpha+4}{2\alpha+3}$ and $p_1=\frac{4\alpha+4}{2\alpha+1}$.

\begin{thm}\label{thm.mapping.properties}
Assume $\alpha>\beta>-\frac{1}{2}$. Then
\begin{enumerate}
\romannum
\item $S_{1,*}$ is bounded on $L^p(\R_+,d\mu)$ for $p\in(p_0,p_1)$;
\item $S_{2,*}$ is bounded from $L^p(\R_+,d\mu)$ into $L^p(\R,d\mu)+L^2(\R,d\mu)$ for all $p\in(1,2]$;
\item $S_{3,*}$ is bounded on  $L^p(\R_+,d\mu)$ for all $p\in(1,2]$;
\item $S_{4,*}$ is bounded from $L^p(\R_+,d\mu)$ into $L^2(\R,d\mu)$ for all $p\in(p_0,\infty)$.
\end{enumerate}
It thus follows that $S_*$ is bounded from $L^p(\R_+,d\mu)$ into $L^p(\R,d\mu)+L^2(\R,d\mu)$ for all $p\in(p_0,2]$.
\end{thm}

\begin{thm}\label{thm.divergence-at-endpoint}
There exists a compactly supported function $f$ in $L^{p_0}(d\mu)$ with the property that $\{S_Rf(x)\}_{R>1}$ diverges for almost every $x\in\R_+$.
\end{thm}

The part most closely resembling the Euclidean counterpart of the disc multiplier is the piece $S_{1*}f$, where the kernel $K_R(t,r)$ is localized in both $t$ and $r$.  The remaining three pieces of $S_*f$ all derive their  mapping properties, to some extend, from the Kunze--Stein phenomenon, perhaps most clearly seen in $S_{3*}$. Philosophically, the localized part $S_1$ of $S_R$ (with $R=1$ for purpose of analogy) should correspond to the Euclidean disc multiplier acting on radial functions. Since the Euclidean disc multiplier is merely $L^2$-bounded when acting on functions not necessarily radial, we cannot except $S_1$ to be bounded on $L^p(G/K)$ unless $p=2$. Note that the full operator $S_R$ ($R=1$) is unbounded on $L^p(K\backslash G/K)$ for $p\neq 2$, since the corresponding multiplier cannot be analytically continued to the strip in the complex plane described in \cite{Clerc-Stein}.
\medskip

For the next result we must first recall the definition of Lorentz spaces.

\begin{defin}
Let $(X,\mu)$ be a measure space, $0<p<\infty$, and $0<q\leq \infty$. By the Lorentz space $L^{p,q}(X,\mu)$ we understand the space of equivalence classes of measurable functions $f$ with finite Lorentz space norm, $\|f\|_{L^{p,q}}<\infty$. Here
\[\|f\|_{L^{p,q}}=\begin{cases} \displaystyle\Bigl(q\int_0^\infty [t\mu(\{x\in X\,:\,\vert f(x)\vert>t\})^{1/p}]^q\,t^{-1}\,dt \Bigr)^{1/q}&\text{if }q<\infty\\
\sup_{t>0}t^{1/p}\mu(\{x\in X\,:\,\vert f(x)\vert>t\})&\text{if }q=\infty.\end{cases}\]
\end{defin}
See \cite[Chapter~1]{Grafakos} for a summary of the properties of Lorentz spaces.

\begin{thm}\label{thm.endpoint-Lorentz}
\begin{enumerate}
\romannum
\item The maximal operator
$S_{1,*}$ is bounded from $L^{p_i,1}(\R_+,d\mu)$ into $L^{p_i,\infty}(\R_+,d\mu)$, $i=0,1$, where $p_0=\frac{4\alpha+4}{2\alpha+3}$ and $p_1=\frac{4\alpha+4}{2\alpha+1}$.
\item The maximal operator $S_{4,*}$ is bounded from $L^{p_0,1}(\R_+,d\mu)$ into $L^2(\R,d\mu)$.
\end{enumerate}
The maximal operator $S_*$ is therefore bounded from $L^{p_0,1}(\R,d\mu)$ into the space $L^2(\R,d\mu)+L^{p_0,\infty}(\R,d\mu)$.
\end{thm}

This was not addressed by Meaney and Prestini but is to be seen as the Jacobi-analysis analogue of the endpoint result in \cite{RomeraSoria}. As for the sharpness of the Lorentz space indices, we mention the following result.

\begin{prop}
The disc multiplier $S_R$ is not bounded from $L^{p_0,r}(\R_+,d\mu)$ into the space $L^{p_0,\infty}(\R,d\mu)+L^2(\R,d\mu)$ for \emph{any} $r\in(1,\infty]$.
\end{prop}
\begin{proof}
The conclusion follows at once from the observation that not even the localized piece $S_R^1$ of the disc multiplier has the stated mapping property, according to \cite[Theorem~II]{Colzani-Hankel}.
\end{proof}

\section{Proof of the Mapping Properties for non-Critical Exponents}
The present section contains the lengthy proof of Theorem \ref{thm.mapping.properties} as well as Theorem \ref{thm.divergence-at-endpoint}. As already indicated one studies $S_*$ in four different regions of the $(t,r)$-plane, so we have split the proof into four subsections.

\subsection{Investigation of $S_{1,*}$}\label{subsec.S1}
We begin the lengthy examination of $S_*$ with an analysis of the behavior of the kernel $K_R$ when both arguments are small. We scale the corresponding operator $S_{1,R}$ slightly by writing
\[S_{1,R}f(t)=\int_{A(t)}K_{1,R}(t,r)f(r)\,d\mu(r),\quad A(t)=\begin{cases} [0,R_0]&\text{if } 0<t\leq \frac{R_0}{2}\\
[0,R_0/2]&\text{if } \frac{R_0}{2}<t<R_0\end{cases}.\] Moreover we assume $2(\alpha+1)$ is \emph{not} an integer since we may copy the proofs from \cite{Meaney-Prestini_invtrans} verbatim, with $n:=2(\alpha+1)$.
Recall from Lemma~\ref{lemma.asymptoticEXP} that $\varphi_\lambda(t)$ may be written as
\[\varphi_\lambda(t)=c\frac{t^{\alpha+\frac{1}{2}}}{\sqrt{\Delta(t)}}\bigl(\mathcal{J}_\alpha(\lambda t)+t^2a_1(1)\mathcal{J}_{\alpha+1}(\lambda t)\bigr)+E_2(\lambda,t),\]
by which
\[\begin{split}
\varphi_\lambda(t)\varphi_\lambda(r)&=c\frac{t^{\alpha+\frac{1}{2}}}{\sqrt{\Delta(t)}} \frac{r^{\alpha+\frac{1}{2}}}{\sqrt{\Delta(r)}}
\bigl(\mathcal{J}_\alpha(\lambda t)\mathcal{J}_\alpha(\lambda r) + \mathcal{J}_\alpha(\lambda t) r^2a_1(r)\mathcal{J}_{\alpha+1}(\lambda r)\\
&\qquad+t^2a_1(t)\mathcal{J}_{\alpha+1}(\lambda t)\mathcal{J}_\alpha(\lambda r) + t^2a_1(t)\mathcal{J}_{\alpha+1}(\lambda t) r^2a_1(r)\mathcal{J}_{\alpha+1}(\lambda r)\bigr)\\
&\qquad+c\frac{t^{\alpha+\frac{1}{2}}}{\sqrt{\Delta(t)}}\bigl(\mathcal{J}_\alpha(\lambda t)+t^2a_1(t)\mathcal{J}_{\alpha+1}(\lambda t)\bigr)E_2(\lambda,r)\\
&\qquad+c\frac{r^{\alpha+\frac{1}{2}}}{\sqrt{\Delta(r)}}\bigl(\mathcal{J}_\alpha(\lambda r)+r^2a_1(r)\mathcal{J}_{\alpha+1}(\lambda r)\bigr)E_2(\lambda,t)\\
&=c\frac{t^{\alpha+\frac{1}{2}}}{\sqrt{\Delta(t)}} \frac{r^{\alpha+\frac{1}{2}}}{\sqrt{\Delta(r)}}\biggl(c_1\frac{J_\alpha(\lambda t)}{(\lambda t)^\alpha}\frac{J_\alpha(\lambda r)}{(\lambda r)^\alpha} + c_2r^2a_1(r)\frac{J_\alpha(\lambda t)}{(\lambda t)^\alpha}\frac{J_{\alpha+1}(\lambda r)}{(\lambda r)^{\alpha+1}}\\
&\qquad +c_3t^2a_1(t)\frac{J_{\alpha+1}(\lambda t)}{(\lambda t)^{\alpha+1}}\frac{J_\alpha(\lambda r)}{(\lambda r)^\alpha}+ r^2t^2a_1(t)a_1(r)\frac{J_{\alpha+1}(\lambda t)}{(\lambda t)^{\alpha+1}}\frac{J_{\alpha+1}(\lambda r)}{(\lambda r)^{\alpha+1}}\biggr)\\
&\qquad+\text{negligible terms}
\end{split}\]
The indicated decomposition yields a compatible decomposition of $K_{1,R}$ and $S_{1,R}f(t)$, in the sense that $K_{1,R}=\sum_{i=1}^5 K_{1,R}^i$ and $S_{1,R}f(t)=\sum_{i=1}^5S_{1,R}^if(t)$, where
\[\begin{split}
K_{1,R}^1(t,r)&=\frac{t^{\alpha+\frac{1}{2}}}{\sqrt{\Delta(t)}}\frac{r^{\alpha+\frac{1}{2}}}{\sqrt{\Delta(r)}}\frac{1}{r^\alpha t^\alpha}\int_0^R\frac{J_\alpha(\lambda r)J_\alpha(\lambda t)}{\lambda^{2\alpha}}\,d\nu(\lambda)\\
K_{1,R}^2(t,r)&=r^2a_1(r)\frac{t^{\alpha+\frac{1}{2}}}{\sqrt{\Delta(t)}}\frac{r^{\alpha+\frac{1}{2}}}{\sqrt{\Delta(r)}}\frac{1}{r^{\alpha+1}t^\alpha}\int_0^R\frac{J_{\alpha+1}(\lambda r)J_\alpha(\lambda t)}{\lambda^{2\alpha+1}}\,d\nu(\lambda)\\
K_{1,R}^3(t,r)&=t^2a_1(t) \frac{t^{\alpha+\frac{1}{2}}}{\sqrt{\Delta(t)}}\frac{r^{\alpha+\frac{1}{2}}}{\sqrt{\Delta(r)}}\frac{1}{r^\alpha t^{\alpha+1}}\int_0^R\frac{J_\alpha(\lambda r)J_{\alpha+1}(\lambda t)}{\lambda^{2\alpha+1}}\,d\nu(\lambda)\\
K_{1,R}^4(t,r)&=r^2a_1(r)t^2a_1(t)\frac{t^{\alpha+\frac{1}{2}}}{\sqrt{\Delta(t)}}\frac{r^{\alpha+\frac{1}{2}}}{\sqrt{\Delta(r)}} \frac{1}{r^{\alpha+1}t^{\alpha+1}} \int_0^R\frac{J_{\alpha+1}(\lambda r)J_{\alpha+1}(\lambda t)}{\lambda^{2(\alpha+1)}}\,d\nu(\lambda)\\
K_{1,R}^5&=\text{sum of negligible terms}
\end{split}\]
(it  is to be understood that all functions are extended by zero to all of $\R$ for $t$ not in $[0,R_0]$), and
\[S_{1,R}^if(t)=\begin{cases}\displaystyle \int_{A(t)}K_{1,R}^i(t,r)f(r)\Delta(r)\,dr&\text{ for } 0\leq t\leq R_0, i=1,\ldots,5\\
0&\text{otherwise}\end{cases}\]
A slightly more convenient expression for $K_{1,R}^1(t,r)$ is obtained by writing
\[\begin{split}
K_{1,R}^1(t,r)&=\frac{t^{\alpha+\frac{1}{2}}}{\sqrt{\Delta(t)}}\frac{r^{\alpha+\frac{1}{2}}}{\sqrt{\Delta(r)}} \frac{1}{r^\alpha}\frac{1}{t^\alpha}
\biggl\{\int_1^R\frac{J_\alpha(\lambda r)J_\alpha(\lambda t)}{\lambda^{2\alpha}}\,d\nu(\lambda) +
\int_0^1\frac{J_\alpha(\lambda r)J_\alpha(\lambda t)}{\lambda^{2\alpha}}\,d\nu(\lambda)\biggr\}\\
&=M_{1,R}(t,r)+E_1(t,r),
\end{split}\]
where
\[\begin{split}
\vert E_1(t,r)\vert&\lesssim \Bigl|\frac{t^{\alpha+\frac{1}{2}}}{\sqrt{\Delta(t)}}\frac{r^{\alpha+\frac{1}{2}}}{\sqrt{\Delta(r)}} \frac{1}{r^\alpha}\frac{1}{t^\alpha}\Bigr|
\int_0^1 \frac{\vert\lambda r\vert^\alpha\vert\lambda t\vert^\alpha}{\lambda^{2\alpha}}\,d\lambda\lesssim 1
\end{split}\] The associated operator
${_e}S_{1,R}^1f(t)=1_{[0,R_0]}(t)\int_{A(t)}E_1(t,r)f(r)\Delta(r)\,dr$ is therefore easily estimated. It holds that
\[\vert{_e}S_{1,R}^1f(t)\vert\leq \|1_{[0,R_0]}E_1(t,\cdot)\|_{L^{p'}(d\mu)} \|1_{[0,R_0]}f\|_{L^p(d\mu)} \lesssim \|f\|_{L^p(d\mu)},\] where $1/p+1/p'=1$, $1<p<\infty$, so the associated maximal operator $t\mapsto \sup_{R>1}\vert{_e}S_{1,R}^1f(t)\vert$ is bounded on $L^p(d\mu)$ for $1<p<\infty$. The term $M_{1,R}(t,r)$ in our kernel decomposition $K_{1,R}^1(t,r)=M_{1,R}(t,r)+E_1(t,r)$ turns out to be fairly complicated, however.

Recall that
\[M_{1,R}(t,r)=\frac{t^{\alpha+\frac{1}{2}}}{\sqrt{\Delta(t)}}\frac{r^{\alpha+\frac{1}{2}}}{\sqrt{\Delta(r)}} \frac{1}{r^\alpha}\frac{1}{t^\alpha} \int_1^RJ_\alpha(\lambda r)J_\alpha(\lambda t)\lambda^{-2\alpha}\vert\cfct(\lambda)\vert^{-2}\,d\lambda.\]
We need a description of not only the leading term in the asymptotic behavior of $\vert\cfct(\lambda)\vert^{-2}$ as $\lambda\to\infty$, so we use Lemma \ref{lemma.precise-c} with $M=\llcorner 2\alpha+2\lrcorner$ (the integer part of $2\alpha+2$) to write $\vert\cfct(\lambda)\vert^{-2} =\lambda^{2\alpha+1}+c_1\lambda^{2\alpha}+c_2\lambda^{2\alpha-1}+\cdots+c_{M-1}\lambda^{2(\alpha+1)-\llcorner{2(\alpha+1)\lrcorner}} + O(\lambda^{2\alpha+1-\llcorner2(\alpha+1)\lrcorner})$. Note that $2\alpha+1<M<2(\alpha+1)$, since $2(\alpha+1)$ is not an integer. Correspondingly the asymptotic expansion of $\vert\cfct(\lambda)\vert^{-2}$ still takes the form $\vert\cfct(\lambda)\vert^{-2}=P(\lambda)+E(\lambda)$, but $P$ is not a polynomial anymore. At any rate
\begin{multline}\label{eqn.decompose-M.R1}
M_{1,R}(t,r)=\frac{t^{\alpha+\frac{1}{2}}}{\sqrt{\Delta(t)}}\frac{r^{\alpha+\frac{1}{2}}}{\sqrt{\Delta(r)}} \frac{1}{r^\alpha}\frac{1}{t^\alpha}\Biggl\{\int_1^RJ_\alpha(\lambda r)J_\alpha(\lambda t)\Bigl(\lambda+c_1+\frac{c_2}{\lambda}+\cdots+\frac{c_{M-1}}{\lambda^{M-2}}\Bigr)\,d\lambda \\
+ \int_1^RJ_\alpha(\lambda r)J_\alpha(\lambda t)E(\lambda)\,d\lambda\Biggr\}
\end{multline}

We thus need to consider separately a host of new operators, like
\[\begin{split}
M_{1,R}^d(t,r)&= \frac{t^{\alpha+\frac{1}{2}}}{\sqrt{\Delta(t)}}\frac{r^{\alpha+\frac{1}{2}}}{\sqrt{\Delta(r)}} \frac{1}{r^\alpha}\frac{1}{t^\alpha}\int_1^RJ_\alpha(\lambda r)J_\alpha(\lambda t)\lambda^d\,d\lambda,\quad d=2-M,\ldots,1\\
&=\frac{t^{\alpha+\frac{1}{2}}}{\sqrt{\Delta(t)}}\frac{r^{\alpha+\frac{1}{2}}}{\sqrt{\Delta(r)}} \frac{1}{r^\alpha}\frac{1}{t^\alpha}\biggl\{\int_0^R\cdots\lambda^d\,d\lambda - \int_0^1\cdots\lambda^d\,d\lambda\biggr\}\end{split}\]
where the latter piece $\cdots\int_0^1\cdots\lambda\,d\lambda$ -- at least for $d=1$ -- gives rise to an operator comparable with ${_e}S_{1,R}^1f$ considered above. We will keep the constants $c_k$, as they never influence the estimates. The conclusion, as before, is that the `error' term gives rise to a maximal operator that is bounded for the full range $p\in(1,\infty)$, hence uninteresting as far as the ongoing proof is concerned.

The first piece in the above-mentioned decomposition of $M_{1,R}^1$ gives rise to the operator
\[\begin{split}
S_{1,R}^{M_{1,R}^1}f(t) & = \int_{A(t)}M^1_{1,R}(t,r)f(r)\Delta(r)\,dr\\
&=\frac{t^{\alpha+\frac{1}{2}}}{\sqrt{\Delta(t)}}\frac{1}{t^\alpha}\int_0^{R_0}\biggl\{\int_0^R J_\alpha(\lambda r)J_\alpha(\lambda t)\lambda\,d\lambda\biggr\} f(r)\sqrt{\Delta(r)}r^{1/2}\,dr.
\end{split}\]
Since $0\leq r,t\leq R_0$, we can introduce the approximation $\sqrt{\Delta(r)}\thicksim r^{\alpha+\frac{1}{2}}$, $\sqrt{\Delta(t)}\thicksim t^{\alpha+\frac{1}{2}}$ in order to arrive at the favorable estimate
\[\begin{split}
\vert S_{1,R}^{M^1_{1,R}}f(t)\vert&\lesssim\frac{1}{t^\alpha}\Bigl| \int_0^{R_0}\biggl\{\int_0^R J_\alpha(\lambda r)J_\alpha(\lambda t)\lambda\,d\lambda\biggr\}f(r)r^{\alpha+1}\,dr\Bigr|\\
&\leq \frac{1}{t^\alpha}\Bigl| \int_0^{\infty}\biggl\{\int_0^R J_\alpha(\lambda r)J_\alpha(\lambda t)\lambda\,d\lambda\biggr\}f(r)r^{\alpha+1}\,dr \Bigr|\\
&=cT_Rf(t),
\end{split}\]
where $T_R$ is formally the  standard partial sums operator for Euclidean Fourier integrals, $f$ being viewed as a radial function on $\R^n$, with $n:=2(\alpha+1)$. More precisely, $T_R$ is indeed  the partial sums operator for the \emph{Hankel transform} $\mathcal{H_\alpha}$, and by  the maximal operator $t\mapsto \sup_{R>1}\vert S_{1,R}^{M_{1,R}^1}f(t)\vert$ is bounded on $L^p(\R_+,x^{2\alpha+1}dx)$ for $\frac{4\alpha+4}{2\alpha+3}<p<\frac{4\alpha+4}{2\alpha+1}$. -- this also explains the appearance of the Herz range. The result of Kanjin applies to $S_{1,R}^{M_{1,R}^1}$ as well, since we have localized its integral kernel in both arguments. We must still consider the remaining pieces $S_{1,R}^{M_{1,R}^0}f(t), S_{1,R}^{M_{1,R}^{-1}}f(t),\ldots$ and $S_{1,R}^{M_{1,R}^E}f(t)$.

Next on the list is the piece (of the kernel $M_{1,R}$)
\[M_{1,R}^0(t,r)=c\frac{t^{\alpha+\frac{1}{2}}}{\sqrt{\Delta(t)}}\frac{1}{t^\alpha} \frac{r^{\alpha+\frac{1}{2}}}{\sqrt{\Delta(r)}}\frac{1}{r^\alpha}\int_1^R J_\alpha(\lambda t)J_\alpha(\lambda r)\,d\lambda.\]
Recall that we DO allow the possibility that $r\geq t$ in this analysis. So let us assume that $r\geq t$ and choose a smooth partition of unity $1=g_{(1)}+g_{(2)}+g_{(3)}$ on $(0,\infty)$, indicated schematically by
\[[1,R]=\underset{(1)}{[1,1/r]}\cup\underset{(2)}{[1/r,1/t]}\cup\underset{(3)}{[1/t,R]}\]
where only the piece $g_{(3)}=:g_t$ will be important. Here $g_t$ is taken to be identically $1$ on $(1/t,R]$ and supported in $(1/r,R]$ (we will later choose $g_t$ more carefully). The corresponding expansion of $M_{1,R}^0$ shall be written
\begin{equation}\label{eqn.MR0}
\begin{aligned}
M_{1,R}^0(t,r) &= c\frac{t^{\alpha+\frac{1}{2}}}{\sqrt{\Delta(t)}}\frac{r^{\alpha+\frac{1}{2}}}{\sqrt{\Delta(r)}}\frac{1}{t^\alpha} \frac{1}{r^\alpha} \Biggl\{\int_1^{1/r}\cdots d\lambda + \int_{1/r}^{1/t}\cdots d\lambda+\int_{1/t}^R\cdots g_t(\lambda)\,d\lambda\Biggr\}\\
&=: M_{1,R}^{0,(1)}(t,r)+M_{1,R}^{0,(2)}(t,r)+M_{1,R}^{0,(3)}(t,r)
\end{aligned}\end{equation}

We first analyze the range $1\leq\lambda\leq 1/r$, corresponding to the function $M_{1,R}^{0,(1)}$. Here we estimate according to
\[\begin{split}
\vert M_{1,R}^{0,(1)}(t,r)\vert&\lesssim \frac{1}{t^\alpha r^\alpha}\int_1^{1/r}\vert J_\alpha(\lambda t)\vert\vert J_\alpha(\lambda r)\vert\, d\lambda\lesssim \frac{1}{t^\alpha r^\alpha}\int_1^{1/r}(\lambda t)^{-1/2}\,d\lambda \\ &=\frac{1}{t^{\alpha+\frac{1}{2}}r^\alpha}\bigl[\lambda^{1/2}\bigr]^{1/r}_1
=\frac{1}{t^{\alpha+\frac{1}{2}}r^\alpha}\Bigl(\frac{1}{r^{1/2}}-1\Bigr) \lesssim \frac{1}{t^{\alpha+\frac{1}{2}}r^{\alpha+\frac{1}{2}}},
\end{split}\]
where we have used that $1/r\geq 1$. It thus follows by the H\"older inequality (with $1/p+1/p'=1$) that

\[\begin{split}
\vert S_{1,R}^{M_{1,R}^{0,(1)}}f(t)\vert&\lesssim \frac{1}{t^{\alpha+\frac{1}{2}}}\int_0^{R_0}\frac{\vert f(r)\vert}{r^{\alpha+\frac{1}{2}}}\Delta(r)\,dr \\
&=\frac{1}{t^{\alpha+\frac{1}{2}}}\Bigl(\int_0^{R_0}\vert f(r)\vert^p\Delta(r)\,dr\Bigr)^{1/p}\Bigl(\int_0^{R_0}r^{-(\alpha+\frac{1}{2})p'}\Delta(r)\,dr\Bigr)^{1/p'}\\
&\lesssim\frac{1}{t^{\alpha+\frac{1}{2}}}\|f\|_{L^p(d\mu)} \Bigl(\int_0^{R_0}r^{2\alpha+1-(\alpha+\frac{1}{2})p'}\,dr\Bigr)^{1/p'}
\end{split}\]
which is finite precisely when $2\alpha+1-(\alpha+\frac{1}{2})p'>-1$, that is, when $\frac{4\alpha+4}{2\alpha+1}>p'$ or equivalently when $p>\frac{4\alpha+4}{2\alpha+3}$. The associated maximal operator $S_{1,*}^{M_{1,R}^{0,(1)}}$ therefore satisfies the estimate $\vert S_{1,*}^{M_{1,R}^{0,(1)}}f(t)\vert\lesssim t^{-(\alpha+\frac{1}{2})}\|_{L^p}$ for $p>\frac{4\alpha+4}{2\alpha+3}$, and $0\leq t\leq R_0$. It follows that
\[\|S_{1,*}^{M_{1,R}^{0,(1)}}f\|_{L^p}^p \lesssim \|f\|_{L^p}^p\int_0^{R_0}t^{-(\alpha+\frac{1}{2})p}\Delta(t)\,dt \lesssim \|f\|_{L^p}^p\int_0^{R_0}t^{-(\alpha+\frac{1}{2})p+2\alpha+1}\,dt,\]
which is finite precisely when $-(\alpha+\frac{1}{2})p+2\alpha+1>-1$, that is, when $p<\frac{4\alpha+4}{2\alpha+1}$. We have therefore established that $S_{1,*}^{M_{1,R}^{0,(1)}}$ is bounded on $L^p$ precisely when $\frac{4\alpha+4}{2\alpha+3}<p<\frac{4\alpha+4}{2\alpha+1}$.
\medskip

The range $1/r\leq\lambda\leq 1/t$, corresponding to the piece $M_{1,R}^{0,(2)}$, is just as easily handled: The standard estimates $\vert J_\mu(\lambda r)\vert\lesssim (\lambda r)^{-1/2}$ and $\vert J_\mu(\lambda t)\vert\lesssim 1$ imply that
\[\vert M_{1,R}^{0,(2)}(t,r)\vert\lesssim \frac{1}{t^\alpha r^\alpha}\int_{1/r}^{1/t}\frac{1}{\lambda^{1/2}}\,d\lambda\lesssim \frac{1}{t^{\alpha+\frac{1}{2}} r^{\alpha+\frac{1}{2}}}.\]
Prior analysis shows that the associated maximal operator $S_{1,*}^{M_{1,R}^{0,(2)}}$ is $L^p$-bounded for $\frac{4\alpha+4}{2\alpha+3}<p<\frac{4\alpha+4}{2\alpha+1}$
\medskip

It turns out to be more difficult to estimate the piece $M_{1,R}^{0,(3)}$, corresponding to the range $1/t\leq\lambda\leq R$. In order to get started we use the more precise Bessel function estimate
\[J_\alpha(t)=c\frac{\cos\bigl(t-\frac{\alpha\pi}{2}-\frac{\pi}{4}\bigr)}{t^{1/2}}+O(t^{-3/2})\]
from \cite[p.~199]{Watson} to write
\[\begin{split}
M_{1,R}^{0,(3)}(t,r) &= c\frac{t^{\alpha+\frac{1}{2}}}{\sqrt{\Delta(t)}}\frac{r^{\alpha+\frac{1}{2}}}{\sqrt{\Delta(r)}}\frac{1}{t^\alpha} \frac{1}{r^\alpha} \int_{1/t}^R\frac{\cos\bigl(\lambda t-\frac{2\alpha+1}{4}\pi\bigr)}{(\lambda t)^{1/2}}\frac{\cos\bigl(\lambda r-\frac{2\alpha+1}{4}\pi\bigr)}{(\lambda r)^{1/2}}g_t(\lambda)\,d\lambda + E\\
&= c\frac{t^{\alpha+\frac{1}{2}}}{\sqrt{\Delta(t)}}\frac{r^{\alpha+\frac{1}{2}}}{\sqrt{\Delta(r)}}\frac{1}{t^\alpha} \frac{1}{r^\alpha} \int_{1/t}^R\frac{\cos\bigl(\lambda t-\frac{2\alpha+1}{4}\pi\bigr)\cos\bigl(\lambda r-\frac{2\alpha+1}{4}\pi\bigr)}{\lambda}g_t(\lambda)\,d\lambda + E\\
&=c\frac{t^{\alpha+\frac{1}{2}}}{\sqrt{\Delta(t)}}\frac{r^{\alpha+\frac{1}{2}}}{\sqrt{\Delta(r)}}\frac{1}{t^\alpha} \frac{1}{r^\alpha}\sum_{\epsilon,\nu=\pm 1} \int_{1/t}^R\frac{e^{i\lambda\epsilon(t+\nu r)}}{\lambda}g_t(\lambda)\,d\lambda+E
\end{split}\]
where $E$ is an error term. Observe that
\begin{multline*}
\int_0^R\frac{e^{i\lambda\varepsilon(t+\nu r)}}{\lambda}g_t(\lambda)\,d\lambda = \biggl(\lambda\mapsto 1_{[0,R]}(\lambda)\frac{g_t(\lambda)}{\lambda}\biggr)^\wedge(\varepsilon(t+\nu r)) \\
=\Biggl(\!\!\biggl(\lambda\mapsto\frac{g_t(\lambda)}{\lambda}\biggr)^\wedge\star\biggl(x\mapsto\frac{e^{iRx}}{x}\biggr)\!\!\Biggr)(\varepsilon+\nu r),
\end{multline*} where $(\cdots)^\wedge$ designates the \emph{Euclidean} Fourier transform. Let $\sigma_t$ denote the Fourier transform of $\lambda\mapsto\frac{g_t(\lambda)}{\lambda}$ and set $\sigma_t(y)=\sup_t\vert\sigma_t(y)\vert$. Choosing $g_t$ as in \cite[Lemma~1]{Meaney-Prestini_invtrans}, it follows that $\sigma_*$ is Lebesgue integrable on $\R$, so that -- apart from the error term $E$ -- we can estimate $S_{1,*}^{M_{1,R}^{0,(3)}}f(t)$ as follows:

\[\begin{split}
\vert S_{1,*}^{M_{1,R}^{0,(3)}}f(t)\vert &\lesssim \frac{1}{t^{\alpha+\frac{1}{2}}}\sum_{\varepsilon,\nu=\pm 1}\sup_{t,R}\biggl|\int\Bigl(\sigma_t\star\Bigl(x\mapsto\frac{e^{iRx}}{x}\Bigr)\!\!\Bigr)(\varepsilon(t+\nu r))f(r)\sqrt{\Delta(r)}\,dr\biggr|\\
&=\frac{1}{t^{\alpha+\frac{1}{2}}}\sum_{\varepsilon,\nu=\pm 1}\sup_{t,R}\biggl|\int\sigma_t(\varepsilon(t+\nu r))\Bigl(\frac{e^{iRx}}{x}\star(f(x)\sqrt{\Delta(x)})\Bigr)(r)\,dr\biggr|\\
&\leq \frac{1}{t^{\alpha+\frac{1}{2}}}\sum_{\varepsilon,\nu=\pm 1}\int\sigma_*(\varepsilon(t+\nu r))\sup_{R>0}\Bigl|\Bigl(\frac{e^{iRx}}{x}\star(f(x)\sqrt{\Delta(x)})\Bigr)(r)\Bigr|\,dr
\end{split}\]
where we recognize the Carleson operator
\[C(f\sqrt{\Delta})(r)=\sup_{R>0}\Bigl(\frac{e^{iRx}}{x}\star(f(x)\sqrt{\Delta(x)})\Bigr)(r)\]
applied to the function $f\sqrt{\Delta}$. Since convolution with the $L^1$-function $\sigma_*$ is an $L^p$-bounded operation, it follows from the weighted estimates for the Carleson operator, developed in \cite{Prestini-summation} (see also \cite{Prestini-weighted}), that
\begin{equation}\label{eqn.Carleson-estimate}
\|S_{1,*}^{M_{1,R}^{0,(3)}}f\|_{L^p(d\mu)}\leq c\|f\|_{L^p(d\mu)}
\end{equation}
for $\frac{4\alpha+4}{2\alpha+3}<p<\frac{4\alpha+4}{2\alpha+1}$. As for the error term $E$ in our decomposition of $M_{1,R}^{0,(3)}$, we notice that
\[\vert E\vert\leq\frac{1}{r^\alpha}\frac{1}{t^\alpha}\frac{1}{r^{1/2}}\frac{1}{t^{3/2}}\int_{1/t}^R\frac{1}{\lambda^2}\,d\lambda \leq \frac{1}{r^{\alpha+\frac{1}{2}}t^{\alpha+\frac{1}{2}}}\]
since $\frac{1}{r}<\frac{1}{t}$ and $\frac{1}{R}<t$. It thus follows, as above, that the error term $E$ gives rise to a maximal operator that is $L^p$-bounded for the same range of $p$.

It could, however, also happen that $R<\frac{1}{r}$, in which case we do not decompose $[0,R]$ but rather estimate $M_{1,R}^0$ directly: It trivially holds that $\vert M_{1,R}^0(r,t)\vert\leq r^{-(\alpha+\frac{1}{2})}t^{-(\alpha+\frac{1}{2})}$, so once again we end up with a maximal operator that is $L^p$-bounded for the stated range of $p$

The last remaining case is $\frac{1}{r}<R<\frac{1}{t}$, where we use the decomposition $[0,R]=[0,1/r]\dot{\cup}[1/r,R]$ and carry out the exact same type of estimates as before. This completes our analysis of $M_{1,R}^0$.
\medskip

Now consider
\[M_{1,R}^{-1}(t,r)=\frac{t^{\alpha+\frac{1}{2}}}{\sqrt{\Delta(t)}}\frac{r^{\alpha+\frac{1}{2}}}{\sqrt{\Delta(r)}}
\frac{1}{t^\alpha}\frac{1}{r^{\alpha}} \int_1^R\frac{J_\alpha(\lambda t)J_\alpha(\lambda r)}{\lambda}\,d\lambda\]
and assume without of generality that $r\geq t$. Decompose the domain of integration smoothly as
\[[1,R]=\underset{(1)}{[1,1/r]}\cup\underset{(2)}{[1/r,1/t]}\cup\underset{(3)}{[1/t,R]}.\]
The ensuing pieces $M_{1,R}^{-1,(1)}$ and $M_{1,R}^{-1,(2)}$ are trivially estimated since they happen to resemble the error terms appearing in the analysis of $M_{1,R}^0$ above. The final piece $M_{1,R}^{-1,(3)}$ is also easily handled, since
\[\Bigl|\int_{1/t}^R\frac{J_\alpha(\lambda r)J_\alpha(\lambda t)}{\lambda}\,d\lambda\Bigr| \leq \frac{1}{r^{1/2}t^{1/2}}\int_{1/t}^R\frac{1}{\lambda^2}\,d\lambda\lesssim\frac{1}{r^{1/2}t^{1/2}}.\]
For the remaining terms $M_{1,R}^i$, $i=2-M,\ldots,-2$  and $M_{1,R}^E$ we employ the trivial estimate $\vert J_\mu(t)\vert\leq c$ for all $t$ to see that
\[\Bigl|M_{1,R}^{-2}(t,r)+\cdots+M_{1,R}^{2-M}(t,r)+M_{1,R}^E(t,r)\Bigr|\lesssim\frac{1}{t^{\alpha+\frac{1}{2}}}\frac{1}{r^{\alpha+\frac{1}{2}}},\]
hence the associated maximal operator  is $L^p$-bounded for $p\in(p_0,p_1)$. We have hereby finished the proof, for the operator $S_{1,R}^1$.
\bigskip

Regarding $S_{1,R}^2$, we first recall that the associated kernel is
\[K_{1,R}^2(t,r)=\frac{t^{\alpha+\frac{1}{2}}}{\sqrt{\Delta(t)}} \frac{r^{\alpha+\frac{1}{2}}}{\sqrt{\Delta(r)}} \frac{1}{t^\alpha}\frac{a_1(r)r^2}{r^\alpha} \int_0^R \frac{J_{\alpha+1}(\lambda r)J_\alpha(\lambda t)}{\lambda^{2\alpha+1}}\vert\cfct(\lambda)\vert^{-2}\,d\lambda.\]
We once more break up the domain of integration as $[0,R]=[0,1]\cup[1,R]$ and use the basic estimate $\vert J_\mu(t)\vert\leq ct^\mu$ to estimate in the range $\lambda\in[0,1]$ as follows:
\[\Bigl| \int_0^1 \frac{J_{\alpha+1}(\lambda r)J_\alpha(\lambda t)}{\lambda^{2\alpha+1}}\vert\cfct(\lambda)\vert^{-2}\,d\lambda\Bigr|  \lesssim \int_0^1\frac{(\lambda r)^{\alpha+1}(\lambda t)^\alpha}{\lambda^{2\alpha+1}}\,d\lambda\leq r^{\alpha+1}t^{\alpha}.\]
In the range $\lambda\in[1,R]$ we once again use the asymptotic expansion of $\vert\cfct(\lambda)\vert^{-2}$ from Lemma \ref{lemma.precise-c} to decompose the integral
\[\int_1^R \frac{J_{\alpha+1}(\lambda r)J_\alpha(\lambda t)}{\lambda^{2\alpha+1}}\vert\cfct(\lambda)\vert^{-2}\,d\lambda\]
further, giving rise to integrals of the sort encountered in of $M_{1,R}$.

By symmetry in $t$ and $r$, the same estimates hold for $K_{1,R}^3(t,r)$, so it remains to investigate the kernels $K_{1,R}^4$ and $K_{1,R}^5$ (the error term) together with the associated operators $S_{1,R}^4$ and $S_{1,R}^5$. To this end recall that

\[K_{1,R}^4(t,r)=\frac{t^{\alpha+\frac{1}{2}}}{\sqrt{\Delta(t)}}\frac{r^{\alpha+\frac{1}{2}}}{\sqrt{\Delta(r)}} \frac{r^2a_1(r)}{r^{\alpha+1}}\frac{t^2a_1(t)}{t^{\alpha+1}} \int_0^R\frac{J_{\alpha+1}(\lambda r)J_{\alpha+1}(\lambda t)}{\lambda^{2(\alpha+1)}}\vert\cfct(\lambda)\vert^{-2}\,d\lambda.\]
For $\lambda\in[0,1]$ in the domain of integration we may proceed as above, and for $\lambda\in[1,R]$ we once again use Lemma \ref{lemma.precise-c}, leading us to consider terms already analysed for $M_{1,R}^3$.

A favorable estimate for the maximal operator associated with the error term $K_{1,R}^5$ also follow from earlier considerations: If $r>t$, we decompose $[0,R]=[0,1]\cup[1,1/r]\cup[1/r,1/t]\cup[1/t,R]$ smoothly, thereby breaking $K_{1,R}^5$ into four pieces. Using that $\vert E_2(\lambda, t)\lesssim t^4$ for $\vert\lambda t\vert\leq 1$ and $\vert E_2(\lambda,t)\vert\lesssim t^2(\lambda t)^{2-(\alpha+\frac{1}{2})}$ for $\vert\lambda t\vert>1$, cf. Lemma \ref{lemma.asymptoticEXP}, it follows that $\vert K_{1,R}^5(t,r)\vert\lesssim r^{-(\alpha+\frac{1}{2})}t^{-(\alpha+\frac{1}{2})}$. The associated maximal operator $S_{1,*}^5$ is therefore $L^p$-bounded for the usual range $p\in(p_0,p_1)$.

By piecing together all the estimates in the present section we have hereby finally proved that $S_{1,*}$ is bounded on $L^p(\R_+,d\mu)$ if and only of $\frac{4\alpha+4}{2\alpha+3}<p<\frac{4\alpha+4}{2\alpha+1}$. We have also seen that the reason for this restricted range is purely Euclidean. At this stage in the analysis, the curved geometry of the underlying symmetric space isn't strong enough that non-Euclidean phenomena overpower the Euclidean structure.

\subsection{Investigation of $S_{2,*}$}
From now on the analysis of $S_R$ will involve the behavior of the Jacobi function $\varphi_\lambda(t)$ when $t$ tends to infinity, and Lemma \ref{lemma.asymptoticEXP} is not applicable in this region. As in the case of symmetric spaces, this investigation
requires sharp bounds on the $\cfct$-function, a close study of the
Harish-Chandra series for $\varphi_\lambda$, and an analogue of the Gangolli
estimates in the Jacobi setting. Recall that
$\varphi_\lambda(t)=\cfct(\lambda)e^{(i\lambda-\rho)t}\phi_\lambda(t)+\cfct(-\lambda)e^{(-i\lambda-\rho)t}\phi_{-\lambda}(t)$, where we now
formally expand $\phi_\lambda(t)$ as a power series (the
``Harish-Chandra series''),
\[\phi_\lambda(t)=\sum_{k=0}^\infty\Gamma_k(\lambda)e^{-2kt}.\] Since $\phi_\lambda$
is a solution to $\mathcal{L}_{\alpha,\beta}\varphi+(\lambda^2+\rho^2)\varphi=0$, the $\Gamma_k(\lambda)$ are
given recursively -- according to \cite[Formula~3.4]{Stanton-Tomas} -- by
$\Gamma_0(\lambda)\equiv 1$,
\begin{multline*}
(k+1)(k+1-i\lambda)\Gamma_{k+1}(\lambda) = (\alpha-\beta)\sum_{j=0}^k(\rho+2j-i\lambda)\Gamma_j(\lambda) \\
+  (\beta+\tfrac{1}{2})\sum_{j=1}^{[\frac{k+1}{2}]}(\rho+2(k+1-2j)-i\lambda)\Gamma_{k+1-2j}(\lambda),
\end{multline*}
where $[\frac{k+1}{2}]$ is the integer part of $\frac{k+1}{2}$. In fact,
$\Gamma_{k+1}=a_k\Gamma_k+\sum_{j=0}^{k-1}b_j^k\Gamma_j$, where (by \cite[Corollary~3.4]{Stanton-Tomas})
\[a_k=1+\frac{\alpha-\beta-1}{k+1}+
\frac{\alpha-\beta-1+\frac{1}{k+1}(\alpha(\alpha-1)-\beta(\beta-1)+1)}{k+1-i\lambda}
\]
and
\[b_j^k=(-1)^{k+j+1}\frac{2\beta+1}{k+1}\Bigl(1+\frac{\rho+2j-1}{k+1-i\lambda}\Bigr).\]

\begin{lem}[Gangolli estimates]
Let $D$ be either a compact subset of $\C\setminus (-i\N)$ or a set of the form
$D=\{\lambda=\xi+i\eta\in\C\,\vert\,\eta\geq -\varepsilon\vert\xi\vert\}$ for
some $\varepsilon\geq 0$. There exist positive constants $K,d$ such that
\begin{equation}\label{eqn.Gangolli.est}
\vert\Gamma_k(\lambda)\vert\leq K(1+k)^d \text{ for all }k\in\Z_+, \lambda\in
D.
\end{equation}
\end{lem}
\begin{proof}
See \cite[Lemma~7]{FJ}.
\end{proof}

It follows that the expansion for $\phi_\lambda(t)$ converges uniformly on sets
of the form $\{(t,\lambda)\in[c,\infty)\times D\}$, where $c$ is a positive
constant. More precisely, if $\lambda\in D$, and $c>0$ is fixed, we see that
\[\forall t\geq c:\vert\phi_\lambda(t)\vert \leq\sum_{k=0}^\infty K(1+k)^de^{-2kt} \lesssim
\sum_{k=0}^\infty (1+k)^de^{-2ck}\lesssim 1,\] that is, $\phi_\lambda(t)$ is
bounded uniformly in $\lambda\in D$ for $t\geq c>0$. We will take $c=R_0$ in
later applications. Since $\lambda\mapsto\phi_\lambda(t)$ is analytic in a
strip containing the real axis, it follows as in the proof of
\cite[Lemma~7]{Meaney-Prestini_invtrans} that derivatives of $\phi_\lambda$ in $\lambda$
are bounded independently of $\lambda$ as well.

\begin{rem}\label{remark.2}
It is easy to prove that $\vert\frac{\partial^k}{\partial\lambda^k}\phi_\lambda(t)\vert\leq c_k$ for all $t\geq R_0$ and $\lambda\in[0,2\rho]$. This was done for symmetric spaces in \cite[Lemma~7]{Meaney-Prestini_invtrans}, whereas a more general statement in the context of Jacobi analysis was obtained in \cite[Lemma~4.1]{Johansen-nonint}.
\end{rem}

The asymptotic behavior of $\varphi_\lambda(t)$ as $t$ increases can now be
investigated. The result is formally the same as the analogues in \cite{Stanton-Tomas} and
\cite{Schindler}, and the proof will even work for complex parameters $\alpha,\beta$.

\begin{thm}\label{thm.bigt}
\begin{enumerate}
\romannum \item For every $M\geq 0$, $0\leq m\leq M$ and $\lambda\in\C$ with $\Im\lambda\geq 0$, there exist polynomials $f_{l_m}$ in $\lambda$ of degree $m$ such that
\[\Gamma_k(\lambda)=\sum_{m=0}^M\gamma_m^k+E_{M+1}^k,\]
where $\gamma_m^k$ is a sum of terms $1/f_{l_m}$, and where
\[\gamma_m^k(\lambda)\vert\leq A\frac{\vert\rho\vert^me^{2k}}{\vert\Re\lambda\vert^m},
\quad \vert D_{\Re\lambda}^a\gamma_m^k\vert\leq 2^aA\frac{\vert\rho\vert^m e^{2k}}{\vert\Re\lambda\vert^{m+a}}, \text{ and }
\vert E_{M+1}^k\vert\leq A\frac{\vert\rho\vert^{M+1}e^{2k}}{\vert\Re\lambda\vert^{M+1}};
\]
the constant $A$ is independent of $M$ and $\lambda$.

\item Let $\Lambda_m(\lambda,t)=\sum_{j=0}^\infty\gamma_m^{m+j}(\lambda)e^{-2jt}$. There exists a function $\mathcal{E}_{M+1}$ such that, for every $M\geq 0$ and $t\geq R_0$, it holds for $\lambda\in\C$ with $\Im\lambda\geq 0$ that
\[\phi_\lambda(t)=\sum_{m=0}^\infty\Lambda_m(\lambda,t)e^{-2mt}=\sum_{m=0}^M\Lambda_m(\lambda,t)e^{-2mt}+e^{-2(M+1)t}\mathcal{E}_{M+1}(\lambda,t),\]
where
\[\vert D_\lambda^a D_t^b\Lambda_m\vert\leq 2^{a+b}A\frac{\vert\rho\vert^me^{2m}}{\vert\Re\lambda\vert^{m+a}}G_b(t)\text{ and } \vert D_t^b\mathcal{E}_{M+1}\vert\leq 2^bA\frac{e^{2(M+1)}\vert\rho\vert^{M+1}}{\vert\Re\lambda\vert^{M+1}}G_b(t),\]
with $G_k(t):=\sum_{j=0}^\infty j^ke^{2k(1-t)}$.
\end{enumerate}
\end{thm}
\begin{proof}
The algebraic properties of the Harish-Chandra series are investigated in \cite[Section~3]{Stanton-Tomas}, along with the estimates in part (i) of the theorem, and it is an arduous (yet elementary) matter to redo the proofs for complex parameters $\alpha,\beta$ instead. The improved statement in (ii) via the presence of the exponential factor in $e^{-2(M+1)t}\mathcal{E}_{M+1}(\lambda,t)$, was established in \cite[Lemma~6]{Meaney-Prestini_invtrans}, the proof of which may trivially be repeated.
\end{proof}
\begin{prop}\label{prop.U}
For $t\geq 1$ we consider the operator
\[Uf(t)=e^{-\rho t}\sup_{R>1}\Biggl| \int_1^\infty\frac{e^{iR(t-r)}}{t-r}f(r)\Delta(r)e^{-\rho r}\,dr\Biggr|.\]
Then $U$ maps $L^p(\R_+,d\mu)$ boundedly into $L^2(\R,d\mu)+L^p(\R,d\mu)$ for $1<p\leq 2$.
\end{prop}
\begin{proof}
The proof is similar to the one for \cite[Theorem~3]{Meaney-Prestini_invtrans} and immediate for $p=2$. For $1<p<2$ fixed and $k\in\N$ write $\varphi_k(t)=1_{[-k,k]}(t)$, $\varphi_k(t)+\psi_k(t)\equiv 1$ for all $t>1$, and correspondingly
\[\begin{split}
e^{-\rho t}\int_1^\infty\frac{e^{iR(t-r)}}{t-r}f(r)\Delta(r)e^{-\rho r}\,dr &= e^{-\rho t}\sum_{k=1}^\infty\int_k^{k+1}\frac{e^{iR(t-r)}}{t-r}f(r)\Delta(r)e^{-\rho r}\,dr\\
&=e^{-\rho t}\sum_{k=1}^\infty(\varphi_k(t)+\psi_k(t))\int_k^{k+1}\frac{e^{iR(t-r)}}{t-r}f(r)\Delta(r)e^{-\rho r}\, dr\\
&=\sum_{k=1}^\infty(A_{k,R}f(t)+B_{k,R}f(t)),
\end{split}\]
where
\[
A_{k,R}f(t)=e^{-\rho t}\varphi_k(t)\int_k^{k+1}\frac{e^{iR(t-r)}}{t-r}f(r)\Delta(r)e^{-\rho r}\,dr\]
and
\[B_{k,R}f(t)=e^{-\rho t}\psi_k(t)\int_k^{k+1}\frac{e^{iR(t-r)}}{t-r}f(r)\Delta(r)e^{-\rho r}\,dr.
\]
Then $\displaystyle Uf(t)\leq\sum_{k=1}^\infty A_{k,*}f(t)+\sup_{R>1}\biggl|\sum_{h=1}^\infty B_{h,R}f(t)\biggr|$,
with
\[
A_{k,*}f(t)=e^{-\rho t}\varphi_k(t)\sup_{R>1}\biggl|\int_k^{k+1}\frac{e^{iR(t-r)}}{t-r}f(r)\Delta(r)e^{-\rho r}\,dr\biggr|
=e^{-\rho t}\varphi_k(t)C(g_k(\cdot))(t),\]
 $C(g_k)$ being the Carleson operator applied to the function $g_k(r)=f(r)\Delta(r)e^{-\rho r}1_{[k,k+1)}(r)$. For $p\in(1,\infty)$ it follows that
\[\begin{split}
\biggl\|\sum_{h=1}^\infty A_{k,*}f\biggr\|^p_{L^p(\R,d\mu)} &\leq \sum_{k=1}^\infty\|A_{k,*}f\|^p_{L^p} = \sum_{k=1}^\infty\bigl\| e^{-\rho t}\varphi_k(t)C(g_k)(t)\bigr\|^p_{L^p}\\
&=\sum_{k=1}^\infty\int_{k-1}^{k+2}e^{-\rho pt}\varphi_k(t)^p(C(g_k)(t))^p\Delta(t)\,dt\\
&\lesssim \sum_{k=1}^\infty e^{-\rho p(k-1)}\int_{k-1}^{k+2}(C(g_k)(t))^p\Delta(t)\,dt\\
(\sharp)&\lesssim\sum_{k=1}^\infty e^{\rho(k-1)(2-p)}\int_{k-1}^{k+2}\vert g_k(r)\vert^p\,dr \\ &=\sum_{k=1}^\infty e^{\rho(k-1)(2-p)}\int_1^\infty\vert g_k(r)\vert^p\,dr\\
&\lesssim\sum_{k=1}^\infty e^{\rho(k-1)(2-p)}\int_k^{k+1}\vert f(r)\vert^pe^{\rho pr}\,dr\\
&=\sum_{k=1}^\infty e^{\rho(k-1)(2-p)}\int_k^{k+1}\vert f(r)\vert^pe^{2\rho r}e^{\rho pr-2\rho r}\,dr
\\ &\thicksim c_p\sum_{k=1}^\infty\int_k^{k+1}\vert f(r)\vert^pe^{2\rho r}\,dr \thicksim c_p\|f\|^p_{L^p}.
\end{split}\]
Here $c_p=e^{\rho(p-2)}$ for $p<2$ and $c_p=e^{2\rho(p-2)}$ for $p\geq 2$.  The actual value of $c_p$ is immaterial of course, but it should be noted that it can become very large. At $(\sharp)$ we have used the classical weighted estimates for the Carleson operator corresponding to the weight $w\equiv 1$.

It now suffices to establish the estimate
\begin{equation}\label{Bk-estimate}
\biggl\|\sup_{R>1}\biggl|\sum_{k=1}^\infty B_{k,R}f(t)\biggr|\biggr\|_{L^2(\R,d\mu)}\leq c_p\|f\|_{L^p(\R,d\mu)}\text{ for } 1<p\leq 2.\end{equation}
For $p=2$ this follows from the easy estimate $\sup_{R>1}\bigl|\sum_{k=1}^\infty B_{k,R}f\bigr|\leq Uf+\sum_{k=1}^\infty A_{k,*}f$, since $U$ is trivially $L^2$-bounded. It remains to establish a restricted $(L^p,L^2)$-estimate for $1<p<2$ by the interpolation theorem of Marcinkiewicz. Fix a measurable subset $E\subset[1,\infty)$. If $\|1_E\|_{L^2(\R,d\mu)}\geq 1$, then $\|\sum_kB_{k,*}1_E\|_{L^2(\R,d\mu)}\lesssim\|1_E\|_{L^2(\R,d\mu)}\lesssim\|1_E\|_{L^p(\R,d\mu)}$ for $p\leq 2$.

On the other hand, if $\|1_E\|_{L^p(\R,d\mu)}<1$,
\[\vert B_{k,*}1_E(t)\vert\leq e^{-\rho t}\biggl(\int_k^{k+1}\frac{1}{\vert t-r\vert}1_E(r)\Delta(r)e^{-\rho r}\,dr\biggr)\psi_k(t)\leq \frac{e^{-\rho t}}{\vert t-h\vert}e^{\rho k}\vert E_k\vert\psi_k(t),\]
where $\vert E_k\vert$ denotes the $\mu$-measure of the set $E_k=E\cap[k,k+1)$. But then
\[\begin{split}
\biggl\|\sum_{k=1}^\infty B_{k,*}1_E\biggr\|_{L^2(\R,d\mu)}&\leq \sum_{k=1}^\infty\|B_{k,*}1_E\|_{L^2(\R,d\mu)}\leq c\sum_{k=1}^\infty e^{\rho k}\vert E_k\vert\\
&\leq c\sum_{k=1}^\infty e^{2\rho k}\vert E_k\vert =c\int_1^\infty 1_E(r)\Delta(r)\,dr\\
&\leq c\|1_E\|^p_{L^p(\R,d\mu)},
\end{split}\]
thus proving the restricted $(L^p,L^2)$-estimate in this case as well. This also finishes the proof of the Proposition.
\end{proof}

\begin{lem}\label{lemma.weighted-estimates}
Let $T$ be either the Carleson operator, the Hilbert transform on $\R$, or a convolution operator with Lebesgue-integrable kernel on $\R$. The maximal operator associated with the function $t\mapsto e^{-\rho t}T(f(r)\Delta(r)e^{-\rho r})(t)$ maps $L^p(\R_+,d\mu)$ boundedly into $L^p(\R,d\mu)+L^2(\R,d\mu)$ for $1<p\leq 2$.
\end{lem}
\begin{proof}
The corresponding statement in \cite{Meaney-Prestini_invtrans} was established by using \cite{Prestini-summation}, hence cast in terms of radial Fourier analysis on $\R^n$. Prestini later generalized her weighted estimates to a setting that applies to Jacobi analysis, witness the paper \cite{Prestini-weighted}. Hence there is nothing new to prove.
\end{proof}

\begin{thm}
The maximal operator $S_{2,*}$ associated with the operator
\[S_{2,R}f(t)=\int_{R_0}^\infty K_{2,R}(t,r)f(r)\Delta(r)\,dr\]
is bounded from $L^p(\R_+,d\mu)$ into  $L^p(\R,d\mu)+L^2(\R,d\mu)$ for $1<p\leq 2$ (where it is implicitly understood that we are in the range $t\geq R_0$).
\end{thm}
The clever technique of proof -- originating in \cite{Meaney-Prestini_invtrans} -- is to bound $S_{2,*}$ by means of maximal operators associated with the Carleson operator and the Hilbert transform, using Proposition \ref{prop.U} and Lemma \ref{lemma.weighted-estimates}. This is a standard technique when working on spaces of homogeneous type, but for weighted measures where the volume of large balls grows exponentially more care is needed. The above technical results are designed to deal with this problem.

\begin{proof}
Adopting an earlier idea we decompose smoothly the domain of integration appearing in the definition of $K_{2,R}$ as $[0,R]=[0,2\rho]\cup[2\rho,R]$ by means of a partition of unity $g_1+g_2\equiv 1$ on $[0,\infty)$, where $g_2\in C_{\text{even}}^\infty(\R)$ is chosen such that $g_2(\lambda)\equiv 1$ for $\lambda>2\rho$, and $\text{supp }g_2\subset[\rho,\infty)$. The kernel $K_{2,R}$ thereby decomposes accordingly as $K_{2,R}=K_{2,R}^E+K_{2,R}^M$, where $K_{2,R}^E$ -- to be regarded as an error term -- is
\[\begin{split}
2K_{2,R}^E(t,r)&=\int_\R e^{(i\lambda-\rho)(r+t)}\phi_\lambda(t)\phi_\lambda(r)g_1(\lambda) \frac{\cfct(\lambda)^2}{\vert\cfct(\lambda)\vert^2}\,d\lambda\\
&\quad + \int_\R e^{i\lambda(t-r)-\rho(t+r)}\phi_\lambda(t)\phi_{-\lambda}(r)g_1(\lambda)\underbrace{\frac{\cfct(\lambda)\cfct(-\lambda)}{\vert\cfct(\lambda)\vert^2}}_{=1}\,d\lambda\\
&\quad + \int_\R e^{i\lambda(r-t)-\rho(t+r)}\phi_{-\lambda}(t)\phi_\lambda(r)g_1(\lambda)\frac{\cfct(-\lambda)\cfct(\lambda)}{\vert\cfct(\lambda)\vert^2}\,d\lambda\\
&\quad + \int_\R e^{(-i\lambda-\rho)(r+t)}\phi_{-\lambda}(t)\phi_{-\lambda}(r)g_1(\lambda)\frac{\cfct(-\lambda)^2}{\vert\cfct(\lambda)\vert^2}\,d\lambda
\end{split}\]
Since $\bigl|\frac{\cfct(\pm\lambda)^2}{\vert\cfct(\lambda)\vert^2}\bigr|=1$, all four terms are estimates in a similar manner, as follows: Let $\mathcal{F}_\lambda$ denotes the Euclidean Fourier transform in $\lambda$. Then
\[\biggl|\int_\R e^{i\lambda x}\phi_{\pm}\lambda(t)\phi_{\pm}\lambda(r)g_1(\lambda)\,d\lambda\biggr| =\biggl| \mathcal{F}_\lambda(\phi_{\pm\lambda}(t)\phi_{\pm\lambda}(r)g_1(\lambda))(x)\biggr|\leq\begin{cases} c&\text{if }\vert x\vert<1\\
1/x^2&\text{if }\vert x\vert>1\end{cases}\]
that is, the functions $\int_\R e^{i\lambda(t-r)}\phi_{\pm\lambda}(t)\phi_{\pm\lambda}(r)g_1(\lambda)\,d\lambda$ are Lebesgue integrable on $\R$ with respect to $t$ and $r$ separately, with $L^1$-norm independent of either $t$ or $r$. Additionally,
\[\Biggl|\int_\R e^{i\lambda(t+r)}\phi_{\pm\lambda}(t)\phi_{\pm\lambda}(r)\frac{\cfct(\mp\lambda)^2}{\vert\cfct(\lambda)\vert^2}g_1(\lambda)\,d\lambda\Biggr|\leq\frac{c}{t+r},\]
so that the maximal operator associated with $K_R^E$ is well behaved.

In order to avoid notational clutter we will now indicate how to proceed with estimates for integrands of the form $\phi_{\pm}(t)\phi_{\pm}(r)e^{i\lambda(t\pm r)}g_1(\lambda)\frac{\cfct(\mp\lambda)^2}{\vert\cfct(\lambda)\vert^2}$; here we allow all possible combinations of signs on $\lambda$ and $r$. Use the first few terms in the Harish-Chandra series expansion for $\phi_\lambda$ to write $\phi_\lambda(t)=\Lambda_0(\lambda,t)+\Lambda_1(\lambda,t)e^{-2t}+\mathcal{E}_2(\lambda,t)e^{-4t}$. Strictly speaking we would obtain $12$ terms in the expansion of $K_{2,R}^M$ upon inserting the Harish-Chandra series for $\phi_\pm(\lambda)$. By abuse of notation we simply write the decomposition of $K_{2,R}^M$ as $K_{2,R}^M=K_{2,R}^{M,0}+K_{2,R}^{M,1}+\mathrm{E}$, where $\mathrm{E}$ is whatever remains. More precisely
\begin{multline*}
K_{2,R}^{M,0}(t,r)=e^{-\rho(t+r)}\Biggl[\int_{-R}^Re^{i\lambda(t-r)}g_2(\lambda)\Lambda_0(\lambda,t)\Lambda_0(\lambda,r)\,d\lambda\\
+\int_{-R}^Re^{i\lambda(t+r)}g_2(\lambda)\Lambda_0(\lambda,t)\Lambda_0(\lambda,r)\frac{\cfct(\mp\lambda)^2}{\vert\cfct(\lambda)\vert^2}\,d\lambda+\text{ similar terms}\Biggr].
\end{multline*}
Since $\Lambda_0(\lambda,t)=1+\sum_{k=1}^\infty\gamma_0^k(\lambda)e^{-2kt}$, one has
\[\begin{split}
\int_{-R}^Re^{i\lambda(t-r)}g_2(\lambda)\,d\lambda &= -\int_\R e^{i\lambda(t-r)}g_1(\lambda)\,d\lambda + \int_{-R}^Re^{i\lambda(t-r)}\,d\lambda\\
&=-\widehat{g_1}(t-r)+\frac{e^{iR(t-r)}}{t-r}-\frac{e^{-iR(t-r)}}{t-r},
\end{split}\]
where $\widehat{g_1}$ is Lebesgue-integrable on $\R$. Proposition \ref{prop.U} is therefore applicable. As for the remaining terms in $\Lambda_0(\lambda,t)\Lambda_0(\lambda,r)$, it is used that $\gamma_0^k(\lambda)$ is in fact a constant, cf. \cite[page~262]{Stanton-Tomas}. Since all these terms decay exponentially fast, their associated maximal operators will be $L^s$-bounded in the full range $1<s<\infty$ and therefore uninteresting.

Moreover, by integration by parts,
\[\begin{split}
\int_{-R}^Re^{i\lambda(t+r)}G_2(\lambda)\,d\lambda & =2\int_\rho^Re^{i\lambda(t+r)}G_2(\lambda)\,d\lambda\\
&=\biggl[\frac{e^{i\lambda(t+r)}}{t+r}G_2(\lambda)\biggr]^R_\rho + \biggl[\frac{e^{i\lambda(t+r)}}{(t+r)^2}G_2'(\lambda)\biggr]_\rho^R \\
&\quad + \frac{1}{(t+r)^2}\int_\rho^R e^{i\lambda(t+r)}G_2''(\lambda)\,d\lambda
\end{split}\]
where $G_2(\lambda)=\frac{\cfct(\mp\lambda)^2}{\vert\cfct(\lambda)\vert^2}g_2(\lambda)\Lambda_0(\lambda,t)\Lambda_0(\lambda,r)$. For $\lambda\geq\rho$ it holds that $\vert G_2(\lambda)\vert\leq c$, $\vert G_2'(\lambda)\vert\leq c/\lambda$, and $\vert G_\lambda''\vert\leq c/\lambda^2$, meaning that
\[\biggl|\int_{-R}^Re^{i\lambda(t+r)}G_2(\lambda)\,d\lambda\biggr|\leq\frac{c}{t+r}.\]

The remaining piece of $K_{2,R}$ is slightly more troublesome, but since the $\gamma_0^k$ in the expansion $\Lambda_0(\lambda,t)=1+\sum_k\gamma_0^k(\lambda)e^{-2kt}$ are constants, we can simplify the investigation at hand by writing
\begin{multline*}
K_{2,R}^{M,1}(t,r)=e^{-\rho(t+r)}\biggl[\int_{-R}^Re^{i\lambda(t-r)}H_-(t,r,\lambda)\,d\lambda\\
+\int_{-R}^Re^{i\lambda(t+r)}H_+(t,r,\lambda)\frac{\cfct(\mp\lambda)^2}{\vert\cfct(\lambda)\vert^2}\,d\lambda+\text{ similar terms}\biggr],
\end{multline*}
with $H_\pm(t,r,\lambda)=\{\Lambda_1(\pm\lambda,r)e^{-2r}+\Lambda_1(\lambda,t)e^{-2t}+\Lambda_1(\pm\lambda,r)\Lambda_1(\lambda,t)e^{-2r}e^{-2t}\}g_2(\lambda)$.
The kernels associated with the indicated three pieces of, say, $H_-$, are all estimated in the same manner, so let us simply consider the first term; it gives rise to the kernel
\[\begin{split}
e^{-2r}\int_{-R}^Re^{i\lambda(t-r)}\Lambda_1(-\lambda,r)g_2(\lambda)\,d\lambda &= e^{-2r}\sum_{j=0}^\infty\int_{-R}^Re^{i\lambda(t-r)}\gamma_1^{1+j}(\lambda)g_2(\lambda)\,d\lambda \\
&= e^{-2r}\sum_{j=0}^\infty e^{-2jr}\left(\widehat{1_{[-R,R]}}* (\gamma_1^{1+j}g_2)^\wedge\right)(t-r),\end{split}\]
where we use the estimate $\vert\gamma_1^{1+j}(\lambda)\vert\leq c(e^{2(1+j)}/\lambda)$ for $\lambda\geq\rho$ to conclude that the Euclidean Fourier transform of $\gamma_1^{1+j}g_2$ is Lebesgue integrable on $\R$, whereas the Fourier transform of $1_{[-R,R]}$ is $\frac{e^{iRx}}{x}$. Again Proposition \ref{prop.U} applies.

The kernel associated with $H_+$ is slightly different, but the kernel associated with each of the three terms in $H_+$ satisfies
\[\biggl|e^{-2r}\int_{-R}^Re^{i\lambda(t+r)}\Lambda_1(\lambda,r)\frac{\cfct(\mp\lambda)^2}{\vert\cfct(\lambda)\vert^2}g_2(\lambda)\,d\lambda\biggr|\leq\frac{e^{-2r}}{t+r}.\]

The final remaining remaining piece $K_{2,R}^M-K_{2,R}^{M,0}-K_{2,R}^{M,1}$ is easily bounded by $\int_\rho^R\lambda^{-2}\,d\lambda\leq c$, valid for all $R>1$, thereby finally completing the proof.
\end{proof}

\subsection{Investigation of $S_{3,*}$}
Recall that we are concerned with the operator
\[S_{3,R}f(t)=\int K_{3,R}(t,r)f(r)\Delta(r)\,dr\]
in the region where $r>R_0$ and $t<\frac{R_0}{2}$.

\begin{lem}\label{lemma.S3.estimate}
For $r>R_0$, $t<\frac{R_0}{2}$, and $R>1$ it holds that
\[\vert K_R(t,r)\vert\lesssim\frac{e^{-\rho t}}{t^{\alpha+\frac{1}{2}}}\frac{1}{r}.\]
\end{lem}
We shall prove the lemma in moment but first we observe that it leads to the desired bound on the relevant maximal function $S_{3,*}$. Indeed, by Lemma \ref{lemma.S3.estimate},
\[\vert S_{3,R}f(t)\vert\lesssim\frac{1}{t^{\alpha+\frac{1}{2}}}\int_{R_0}^\infty\vert f(r)\vert\frac{e^{\rho r}}{r}\,dr \lesssim \frac{\|f\|_{L^p(d\mu)}}{t^{\alpha+\frac{1}{2}}} \int_{R_0}^\infty \frac{e^{\rho r(1-\frac{2}{p})p'}}{r^{p'}}\,dr \leq \frac{c_p}{t^{\alpha+\frac{1}{2}}}\|f\|_{L^p(d\mu)}\]
for $\frac{1}{p}+\frac{1}{p'}=1$, $1<p\leq 2$. By typical arguments we conclude that $\|S_{3,*}f\|_{L^p\R,(d\mu)}\leq c_p\|f\|_{L^p(\R_+,d\mu)}$ for $1<p\leq 2$.

\begin{proof}[Proof of Lemma \ref{lemma.S3.estimate}]
Observing that $\frac{1}{\vert t-r\vert}\lesssim\frac{1}{r}$ and $\frac{1}{t+r}\lesssim\frac{1}{r}$, we decompose the kernel $K_{3,R}(t,s)$ as a sum $K_{3,R}=K_{3,R}^{(1)}+K_{3,R}^{(2)}+K_{3,R}^{(3)}$ where we have put
\[\begin{split}
K_{3,R}^{(1)}(t,r)&=\int_0^{2\rho}\varphi_\lambda(t)\varphi_\lambda(r)\vert\cfct(\lambda)\vert^{-2}\,d\lambda,\\
K_{3,R}^{(2)}(t,r)&=\int_{2\rho}^{1/t}\varphi_\lambda(t)\varphi_\lambda(r)\vert\cfct(\lambda)\vert^{-2}\,d\lambda, \text{ and}\\
K_{3,R}^{(3)}(t,r)&=\int_{1/t}^R\varphi_\lambda(t)\varphi_\lambda(r)\vert\cfct(\lambda)\vert^{-2}\,d\lambda.
\end{split}\]

Here
\[\begin{split}
K_{3,R}^{(1)}(t,r)&=\int_0^{2\rho}\varphi_\lambda(t)\bigl(\cfct(\lambda)e^{(i\lambda-\rho)r}\phi_\lambda(r)+\cfct(-\lambda)e^{(-i\lambda-\rho)r}\phi_{-\lambda}(r)\bigr)\vert\cfct(\lambda)\vert^{-2}\,d\lambda\\
&=e^{-\rho r}\Bigl[\frac{1}{ir}e^{i\lambda r}\frac{\varphi_\lambda(t)\phi_\lambda(r)}{\cfct(-\lambda)}\Bigr]_0^{2\rho}-\frac{1}{ir}e^{-\rho r}\int_0^{2\rho}e^{i\lambda r}\frac{d}{d\lambda}\Bigl(\frac{\varphi_\lambda(t)\phi_\lambda(r)}{\cfct(-\lambda)}\Bigr)\,d\lambda\\
&\quad + e^{-\rho r}\Bigl[-\frac{1}{ir}e^{-i\lambda r}\frac{\varphi_\lambda(t)\phi_{-\lambda}(r)}{\cfct(\lambda)}\Bigr]_0^{2\rho} + \frac{1}{ir}e^{-\rho r}\int_0^{2\rho}e^{-i\lambda r}\frac{d}{d\lambda}\Bigl(\frac{\varphi_\lambda(t)\phi_{-\lambda}(r)}{\cfct(\lambda)}\Bigr)\,d\lambda
\end{split}\]
by integration by parts. By Lemma \ref{lemma.precise-c} and Remark \ref{remark.2}, the derivatives $\frac{d}{d\lambda}\bigl(\frac{\varphi_\lambda(t)\phi_\lambda(r)}{\cfct(-\lambda)}\bigr)$ and $\frac{d}{d\lambda}\bigl(\frac{\varphi_\lambda(t)\phi_{-\lambda}(r)}{\cfct(\lambda)}\bigr)$ are bounded for $0\leq \lambda\leq 2\rho$, so $\vert K_{3,R}^{(1)}(t,r)\vert\lesssim\frac{e^{-\rho r}}{r}$. Note that this estimate is stronger than what was stated in Lemma \ref{lemma.S3.estimate}, since for small $t$ the factor $\frac{1}{t^{\alpha+\frac{1}{2}}}$ would become large, implying a very poor kernel estimate.

Integration by parts, now for $K_{3,R}^{(2)}$, shows that
\[\begin{split}
K_{3,R}^{(2)}(t,r)&=e^{-\rho r}\Bigl[\frac{1}{ir}e^{i\lambda r}\frac{\varphi_\lambda(t)\phi_\lambda(r)}{\cfct(-\lambda)}\Bigr]_{2\rho}^{1/t}-\frac{1}{ir}e^{-\rho r}\int_{2\rho}^{1/t}e^{i\lambda r}\frac{d}{d\lambda}\Bigl(\frac{\varphi_\lambda(t)\phi_\lambda(r)}{\cfct(-\lambda)}\Bigr)\,d\lambda\\
&\quad + e^{-\rho r}\Bigl[-\frac{1}{ir}e^{-i\lambda r}\frac{\varphi_\lambda(t)\phi_{-\lambda}(r)}{\cfct(\lambda)}\Bigr]_{2\rho}^{1/t} + \frac{1}{ir}e^{-\rho r}\int_{2\rho}^{1/t}e^{-i\lambda r}\frac{d}{d\lambda}\Bigl(\frac{\varphi_\lambda(t)\phi_{-\lambda}(r)}{\cfct(\lambda)}\Bigr)\,d\lambda
\end{split}\]
where we utilize the estimates $\vert\varphi_\lambda(t)\vert\leq c$, $\vert\phi_{\pm\lambda}(r)\vert\leq c$, $\vert 1/\cfct(\pm\lambda)\vert\leq c\lambda^{\alpha+\frac{1}{2}}$, $\vert\varphi_\lambda'(t)\vert\leq c/\lambda$, $\vert\phi_{\pm\lambda}'(r)\vert\leq ce^{-2r}/\lambda$, and $\vert\frac{d}{d\lambda}(\varphi_\lambda(t)\phi_{\pm\lambda}(r)/\cfct(\pm\lambda))\vert\leq c\lambda^{\alpha-\frac{1}{2}}$ to conclude that $\vert K_R^{(2)}(t,r)\vert\leq c\frac{e^{-\rho r}}{t^{\alpha+\frac{1}{2}}}\frac{1}{r}$.

It remains to study $K_{3,R}^{(3)}$ but since $\lambda$ is allowed to become either very large (when $R$ is large) or small (less than one, at least), here we have to be slightly more careful in the estimates, especially since the proof in \cite{Meaney-Prestini_invtrans} leaves out most terms in the calculation (similar to what happened in the analysis of $K_{2,R}$). First write $\varphi_\lambda(t)$ as \[\varphi_\lambda(t)=\frac{t^{\alpha+\frac{1}{2}}}{\sqrt{\Delta(t)}}\biggl[\frac{J_{\alpha}(\lambda t)}{(\lambda t)^\alpha}+E_1(\lambda,t)\biggr],\]
so that
\[\begin{split}
K_{3,R}^{(3)}(t,r)&=e^{-\rho r}\frac{t^{1/2}}{\sqrt{\Delta(t)}}\int_{1/t}^R\frac{J_\alpha(\lambda t)}{\lambda^\alpha}\phi_\lambda(r)e^{i\lambda r}\frac{d\lambda}{\cfct(-\lambda)} \\ & + e^{-\rho r}\frac{t^{1/2}}{\sqrt{\Delta(t)}}\int_{1/t}^R\frac{J_\alpha(\lambda t)}{\lambda^\alpha}\phi_{-\lambda}(r)e^{-i\lambda r}\frac{d\lambda}{\cfct(\lambda)}\\
&+ e^{-\rho r}\frac{t^{\alpha+1/2}}{\sqrt{\Delta(t)}}\int_{1/t}^RE_1(\lambda,t)\phi_\lambda(r)e^{i\lambda r}\frac{d\lambda}{\cfct(-\lambda)} \\ &+ e^{-\rho r}\frac{t^{\alpha+1/2}}{\sqrt{\Delta(t)}}\int_{1/t}^RE_1(\lambda,t)\phi_{-\lambda}(r)e^{-i\lambda r}\frac{d\lambda}{\cfct(\lambda)}
\end{split}
\]
where the first two terms are satisfy the same estimates. We therefore concentrate on the first one. To this end recall that $J_\alpha(t)\thicksim t^{-1/2}\cos(t-\frac{2\alpha+1}{4}\pi)+O(t^{-3/2})$ for $t\to\infty$, so that
\[\begin{split}
\int_{1/t}^R \frac{J_\alpha(\lambda)}{\lambda^\alpha}\phi_\lambda(r)e^{i\lambda r}\frac{d\lambda}{\cfct(-\lambda)} &= \gamma_1t^{-1/2}\int_{1/t}^R\frac{e^{i\lambda(r+t)}}{\lambda^{\alpha+\frac{1}{2}}}\phi_\lambda(r)\frac{d\lambda}{\cfct(-\lambda)}\\
&\quad +\gamma_{-1}t^{-1/2}\int_{1/t}^R\frac{e^{i\lambda(r-t)}}{\lambda^{\alpha+\frac{1}{2}}}\phi_\lambda(r)\frac{d\lambda}{\cfct(-\lambda)}\\ &\quad + \int_{1/t}^R\frac{E(\lambda)}{\lambda^{\alpha+\frac{1}{2}}}e^{i\lambda r}\phi_\lambda(r)\frac{d\lambda}{\cfct(-\lambda)}
\end{split}\]
where $\gamma_{\pm 1}=\exp(\pm i\frac{2\alpha+1}{4}\pi)$. Since $E(\lambda)\thicksim O(\lambda^{-3/2})$ and $\cfct(-\lambda)^{-1}\thicksim\lambda^{\alpha+\frac{1}{2}}$, the third integral is easily bounded.
As for the first two integrals it holds that
\begin{multline*}
\int_{1/t}^R\frac{e^{i\lambda(r\pm t)}}{\lambda^{\alpha+\frac{1}{2}}}\phi_\lambda(r)\frac{d\lambda}{\cfct(\mp\lambda)} = \Bigl[\frac{e^{i\lambda(r\pm t)}}{i(r\pm t)}\frac{\phi_\lambda(r)}{\lambda^{\alpha+\frac{1}{2}}\cfct(\mp\lambda)}\Bigr]_{1/t}^R \\ - \frac{1}{i(r\pm t)} \int_{1/t}^Re^{i\lambda(r\pm t)}\frac{d}{d\lambda}\Bigl(\frac{\phi_\lambda(r)}{\lambda^{\alpha+\frac{1}{2}}\cfct(\mp\lambda)}\Bigr)\,d\lambda
\end{multline*}
The first term is dominated by $c/r$, which is what we need, whereas the second term is controlled by an additional integration by parts. This gives rise to the additional contribution
\[\Bigl[\frac{e^{i\lambda(r\pm t)}}{(r\pm t)^2}\frac{d}{d\lambda}\Bigl(\frac{\phi_\lambda(r)}{\lambda^{\alpha+\frac{1}{2}}\cfct(\mp\lambda)}\Bigr)\Bigr]_{1/t}^R + \frac{1}{(r\pm t)^2}\int_{1/t}^Re^{i\lambda(r\pm t)}\frac{d^2}{d\lambda^2}\Bigl(\frac{\phi_\lambda(r)}{\lambda^{\alpha+\frac{1}{2}}\cfct(\mp\lambda)}\Bigr)\,d\lambda,\]
where the first term is bounded by $c/r^2$ and where the integral is bounded by
\[\frac{c}{r^2}\int_{1/t}^R\frac{1}{\lambda^2}\,d\lambda\leq\frac{c}{r^2}.\]
Collecting powers in $t$ (observing that $\sqrt{\Delta(t)}\thicksim t^{\alpha+\frac{1}{2}}$ for $t<R_0$ and that we gained the factor $t^{-1/2}$ when estimating $\int_{1/t}^RJ_\alpha(\lambda t)\lambda^{-\alpha}\phi_\lambda(r)e^{i\lambda r}\cfct(-\lambda)^{-1}d\lambda$, the required kernel estimate drops out. The remaning terms in the decomposition of $K_{3,R}^{(3)}$ is treated analogously.
\end{proof}
\begin{rem}
The proof is as in \cite{Meaney-Prestini_invtrans} but it must be pointed out that the proof in \cite{Meaney-Prestini_invtrans} has a technical gap, in that the authors ignore $\phi_{-\lambda}$ and $\cfct(-\lambda)$ in the estimates. The results on asymptotic properties of $\varphi_\lambda$ and $\vert\cfct(\lambda)\vert^{-2}$ are stated under the assumption that $\lambda$ be nonnegative, so one must be more careful. Moreover, as we do not complex conjugate anywhere, and since the original proof was a bit short, we have filled out the gaps along the way. Lemma \ref{lemma.precise-c} and the results on the asymptotic behavior of $\varphi_\lambda(t)$ have been stated and proved in a way that repairs this small deficiency.
\end{rem}

\subsection{Investigation of $S_{4,*}$}
It remains to analyze $S_{4,*}$ but the required estimates follow at once from those for $S_{3,*}$ once we have interchanged $t$ and $r$. The resulting kernel estimate is
\begin{equation}\label{eqn.K4}
\vert K_R(t,r)\vert\leq c\frac{e^{-\rho t}}{t}\frac{1}{r^{\alpha+\frac{1}{2}}}\text{ for } r<\frac{R_0}{2}\text{ and } t>R_0,
\end{equation}
which implies  that
\[\begin{split}
\vert S_{4,R}f(t)\vert& \lesssim \int_0^{R_0/2}\frac{e^{-\rho t}}{t}\frac{1}{r^{\alpha+\frac{1}{2}}}\vert f(r)\vert \Delta(r)\,dr \\ \lesssim \frac{e^{-\rho t}}{t}\|f\|_{L^p}\Bigl(\int_0^{R_0/2}r^{-(\alpha+\frac{1}{2})p'+2\alpha+1}\,dr\Bigr)^{1/p'},
\end{split}\]
which is finite whenever $p'<\frac{4\alpha+4}{2\alpha+1}$.  Unlike the operator $S_{1,*}$, $L^p$-boundedness of $S_{4,*}$ does not require an additional constraint on the range of $p$. Indeed $\vert S_{4,*}f(t)\vert\lesssim \frac{e^{-\rho t}}{t}\|f\|_{L^p}$, whence
\begin{multline*}
\|S_{4,*}f\|^2_{L^2(d\mu)} = \int_{R_0}^\infty\vert S_{4,*}f(t)\vert^2\Delta(t)\,dt\\ \lesssim \int_{R_0}^\infty\frac{e^{-2\rho t}}{t^2}\|f\|_{L^p}^2\Delta(t)\,dt \lesssim \|f\|_{L^p}^2\int_{R_0}^\infty\frac{dt}{t^2}\lesssim \|f\|_{L^p}^p.\end{multline*}
In other words $\|S_{4,*}f\|_{L^2(d\mu)}\leq c\|f\|_{L^p}$ for $\frac{4\alpha+4}{2\alpha+3}<p$.

\subsection{Divergence at $p=p_0$}
We will presently prove Theorem \ref{thm.divergence-at-endpoint} regarding the existence of a particularly unpleasant function $f\in L^{p_0}(d\mu)$. The technique is an easy extension of the one used to establish \cite[Theorem~4]{Meaney-Prestini_invtrans}, which we review for the sake of completeness. It was already used in \cite{Meaney-divergent}, which in turn was an extension of the classical Cantor--Lebesgue Lemma (for trigonometric series)   to the setting of Jacobi polynomials on $[-1,1]$.

In the following, one should think of the parameter $\alpha$ (which Meaney and Kanjin use for the Hankel transform) as our Jacobi parameter $\alpha$; the point is that we may ignore the other Jacobi parameter $\beta$ when we are merely interested in the local (Euclidean) behavior of the Jacobi functions. So assume $\alpha\geq-\frac{1}{2}$, $p\in[1,\infty)$, and $0\leq a<b\leq\infty$. Let
$L_\alpha^p((a,b))$ denote the space of all measurable functions $g$ on $\R_+$ for which
\[\|g\|_{\alpha,p}=\Bigl(\int_a^b\vert g(t)\vert^pt^{2\alpha+1}\,dt\Bigr)^{1/p}<\infty.\]
\begin{lem}\label{lemma.Kanjin-CL}
Assume $\frac{4\alpha+2}{2\alpha+3}\leq p\leq 2$, and that $F\in L_\alpha^{p'}((1,\infty))$ has the property that
\[\lim_{R\to\infty}\int_R^{R+h}F(\lambda)\Bigl(\frac{J_\alpha(\lambda t)}{(\lambda t)^\alpha}\Bigr)\lambda^{2\alpha+1}\,d\lambda = 0\]
uniformly in $h\in[0,1]$. It then follows that
\[\lim_{R\to\infty}\int_R^{R+h}F(\lambda)\lambda^{\alpha+\frac{1}{2}}\,d\lambda = 0\]
uniformly in $h\in[0,1]$.
\end{lem}
We refer to \cite{Kanjin-summation} for a proof. The Lemma will be applied to $F=\widehat{f}$, which is permissible since the Hausdorff--Young inequality implies that $\|\widehat{f}\|_{L^{p'}(d\nu)}\lesssim\|f\|_{L^p(d\mu)}$ whenever $f$ belongs to $L^p(d\mu)$. Since $\vert\cfct(\lambda)\vert^{-2}\thicksim\lambda^{2\alpha+1}$ for $\lambda\to\infty$, it thus follows that $\widehat{f}\vert_{[1,\infty)}\in L_\alpha^{p'}((1,\infty))$.

\begin{lem}\label{lemma.MP}
Assume $p\in[\frac{4\alpha+2}{2\alpha+3},2]$, and that $f\in L^p(\R_+,d\mu)$ has the property that $\lim_{R\to\infty}S_Rf(t)$ exists for every $t$ in a subset $E\subset[0,1]$ of positive measure. It follows that
\begin{equation}\label{eqn.CL-convergence}
\lim_{R\to\infty}\int_R^{R+h}\widehat{f}(\lambda)\vert\cfct(\lambda)\vert^{-1}\,d\lambda=0
\end{equation}
uniformly in $h\in[0,1]$.
\end{lem}
The proof of Theorem \ref{thm.divergence-at-endpoint} will be completed once we have produced a function $f\in L^{p_0}(\R_+,d\mu)$ that violates the conclusion \eqref{eqn.CL-convergence}.
\begin{proof}
We may assume without loss of generality that $E$ is contained in an interval of the form $[\varepsilon,1]$, $\lambda>1/\varepsilon$. Using that
\[\varphi_\lambda(t)=c\frac{t^{\alpha+\frac{1}{2}}}{\sqrt{\Delta(t)}}\biggl(\frac{J_\alpha(\lambda t)}{(\lambda t)^\alpha} + t^2a_1(t)\frac{J_{\alpha+1}(\lambda t)}{(\lambda t)^{\alpha+1}}\biggr)+E_2(\lambda,t)\]
with $\vert E_2(\lambda,t)\vert\lesssim t^{2-(\alpha+\frac{1}{2})}\lambda^{-\frac{2\alpha+5}{2}}$ for $\vert\lambda t\vert>1$, it suffices (due to the fact that $\vert\cfct(\lambda)\vert^{-2}/\lambda^{2\alpha+1}\lesssim 1$ for $\lambda\to\infty$) to show that \emph{if}
\[\lim_{R\to\infty}\int_R^{R+h}\widehat{f}(\lambda)\varphi_\lambda(t)\vert\cfct(\lambda)\vert^{-2}\,d\lambda =0\]
for $\varepsilon<t<1$, then
\[\lim_{R\to\infty}\int_R^{R+h}\widehat{f}(\lambda)\frac{J_\alpha(\lambda t)}{(\lambda t)^\alpha}\lambda^{2\alpha+1}\,d\lambda=0\]
since Lemma \ref{lemma.Kanjin-CL} is then applicable. Note that the conclusion is not automatic, since the integrands are not always positive. There could be lots of oscillation going on that would prevent the requirement in Lemma \ref{lemma.Kanjin-CL} to be satisfied. We must therefore prove the statements
\begin{align}
 \lim_{R\to\infty}\int_R^{R+h}\widehat{f}(\lambda)\frac{J_{\alpha+1}(\lambda t)}{(\lambda t)^{\alpha+1}}\lambda^{2\alpha+1}\,d\lambda =0 \tag{i}\label{eqn1}\\
\lim_{R\to\infty}\int_R^{R+h}\widehat{f}(\lambda)E_2(\lambda,t)\lambda^{2\alpha+1}\,d\lambda=0\tag{ii}\label{eqn2}
\end{align}
As for \eqref{eqn1} it follows from the usual Bessel function estimate $\vert J_\mu(x)\vert\lesssim x^{-1/2}$ for large $x$ that
\begin{multline*}
\biggl|\int_R^{R+h}\widehat{f}(\lambda)\frac{J_{\alpha+1}(\lambda t)}{(\lambda t)^{\alpha+1}}\lambda^{2\alpha+1}\,d\lambda\biggr|\leq ct^{-\frac{2\alpha+3}{2}}\int_R^{R+h}\vert\widehat{f}(\lambda)\vert\lambda^{-\frac{2\alpha+3}{2}}\lambda^{2\alpha+1}\,d\lambda\\
\leq ct^{-\frac{2\alpha+3}{2}}\Bigl(\int_R^{R+h}\vert\widehat{f}(\lambda)\vert^{p'}\lambda^{2\alpha+1}\,d\lambda\Bigr)^{1/p'} \Bigl(\int_R^{R+h}\lambda^{-\frac{2\alpha+3}{2}p}\lambda^{2\alpha+1}\,d\lambda\Bigr)^{1/p}
\end{multline*}
where the first integral is bounded by $\|\widehat{f}\|_{L_\alpha^{p'}((1,\infty))}$. The second integral is roughly of size $(hR^{-\frac{2\alpha+3}{3}p+2\alpha+1})^{1/p}$, which tends to zero as $R\to\infty$, since the assumption that $p$ be larger that $\frac{4\alpha+2}{2\alpha+3}$ implies that $2\alpha+1-(\alpha+\frac{3}{2})<0$. Therefore \eqref{eqn1} holds; the proof of \eqref{eqn2} is just as easy.
\end{proof}
\begin{proof}[Proof of Theorem \ref{thm.divergence-at-endpoint}]
Let
\[
F_R(f)=\int_R^{R+1}\widehat{f}(\lambda)\vert\cfct(\lambda)\vert^{-1}\,d\lambda = \int_0^1\Bigl\{\int_R^{R+1}\varphi_\lambda(t)\vert\cfct(\lambda)\vert^{-1}\,d\lambda\Bigr\}f(t)\Delta(t)\,dt\]
for $R>0$ and $f\in L^p(d\mu)$ with $\text{supp }f\subset[0,1]$. It is seen that the  operator norm of $F_R$ is precisely
\[\|F_R\|=\Bigl(\int_0^1\Bigl|\int_R^{R+1}\varphi_\lambda(t)\vert\cfct(\lambda)\vert^{-1}\,d\lambda\Bigr|^{p'}\Delta(t)\,dt\Bigr)^{1/p'}\]
which in turn is just the norm of $t\mapsto\int_R^{R+1}\varphi_\lambda(t)\vert\cfct(\lambda)\vert^{-1}\,d\lambda$ in $L^{p'}([0,1],\Delta(t)dt)$. Keeping in mind that $\varphi_\lambda(t)=c\frac{t^{\alpha+\frac{1}{2}}}{\sqrt{\Delta(t)}}\frac{J_\alpha(\lambda t)}{(\lambda t)^\alpha}+E_1(\lambda,t)$ with $\vert E_1(\lambda,t)\vert\lesssim t^2(\lambda t)^{-\frac{2\alpha+3}{2}}$ for $\vert\lambda t\vert>1$, we infer from the proof of \cite[Lemma~2]{Kanjin-summation} (see also the proof of \cite[Lemma~1]{Kanjin}) that
\begin{multline*}
\biggl\|t\mapsto \int_R^{R+1}\frac{J_\alpha(\lambda t)}{(\lambda t)^\alpha}\lambda^{\alpha+\frac{1}{2}}\,d\lambda\biggr\|_{L^{p'}([0,1],\Delta(t)dt)} \\
\thicksim \biggl\|t\mapsto \int_R^{R+1}\frac{J_\alpha(\lambda t)}{(\lambda t)^\alpha}\lambda^{\alpha+\frac{1}{2}}\,d\lambda\biggr\|_{L^{p'}([0,1],t^{2\alpha+1}dt)}\succsim (\log R)^{1/p'}.
\end{multline*}
Moreover
\[\int_{1/R}^1\Bigl|\int_R^{R+1}E_1(\lambda,t)\lambda^{\alpha+\frac{1}{2}}\,d\lambda\Bigr|^{p'}\Delta(t)\,dt\]
is uniformly bounded in $R$, when we take $p=p_0$. Indeed,
\[\begin{split}
\int_{1/R}^1\Bigl|\int_R^{R+1}E_1(\lambda,t)\lambda^{\alpha+\frac{1}{2}}\,d\lambda\Bigr|^{p'}\Delta(t)\,dt &\leq \int_{1/r}^1\Bigl|\int_R^{R+1}c_1t^2(\lambda t)^{-\frac{2\alpha+3}{2}}\lambda^{\alpha+\frac{1}{2}}\,d\lambda\Bigr|^{p'}\Delta(t)\,dt\\
&\leq \int_{1/R}^1\Bigl|\int_R^{R+1}c_1\lambda^{-1}\,d\lambda\Bigr|^{p'}t^{p'(2-\frac{2\alpha+3}{2})}t^{2\alpha+1}\,dt\\
&=c\Bigl(\log\frac{R+1}{R}\Bigr)^{p'}\int_{1/R}^1 t^{\frac{2\alpha+3}{2\alpha+1}}\,dt \\
&= c'\Bigl(\log\frac{R+1}{R}\Bigr)^{p'}\bigl[t^{\frac{4\alpha+4}{2\alpha+1}}\bigr]^1_{1/R}\\
&=c'\Bigl(\log\frac{R+1}{R}\Bigr)^{\frac{4\alpha+4}{2\alpha+1}}R^{-\frac{4\alpha+4}{2\alpha+1}} =o(1)
\end{split}\] for $R\to\infty$. By the Banach--Steinhaus theorem there exists a function $f\in L_\alpha^{p_0}((0,1))$ so that
\[\limsup_{R\to\infty}\Bigl|\int_R^{R+1}\widehat{f}(\lambda)\vert\cfct(\lambda)\vert^{-1}\,d\lambda\Bigr|=0.\]
It thus follows from Lemma \ref{lemma.MP} that $\{S_Rf(t)\}_R$ diverges for almost every $t\in[0,1]$.
\end{proof}

\section{Proof of the Mapping Properties for Critical Exponents}
We now prove Theorem \ref{thm.endpoint-Lorentz}. Since $S_{2,*}$ and $S_{3,*}$ do not behave worse on $L^{p_0}$ than on other $L^p$-spaces, it suffices to establish the endpoint mapping properties of $S_{1,*}$ and $S_{4,*}$.  The endpoint mapping property of $S_{4,*}$ is stated below as Lemma \ref{lemma.S4-endpoint}, so we shall presently concentrate on $S_{1,*}$.

Recall from Subsection \ref{subsec.S1} that we decomposed the integral kernel $K_{R,1}$ for the localized piece $S_{1,R}$ of the disc multiplier into a large collection of pieces. The contributions $K_{1,R}^2,\ldots, K_{1,R}^5$ are easily handled, so we begin with those:

As for $K_{R,1}^1$, we introduced a further decomposition $K_{R,1}^1(t,r)=M_{1,R}(t,r)+E_1(t,r)$, where $M_{1,R}$ was decomposed even further, cf. \eqref{eqn.decompose-M.R1}, into functions of the form
\[M_{1,R}^d(t,r)= \frac{t^{\alpha+\frac{1}{2}}}{\sqrt{\Delta(t)}}\frac{r^{\alpha+\frac{1}{2}}}{\sqrt{\Delta(r)}} \frac{1}{r^\alpha}\frac{1}{t^\alpha}\int_1^RJ_\alpha(\lambda r)J_\alpha(\lambda t)\lambda^d\,d\lambda,\quad d=2-M,\ldots,1\]
The operator $S_{1,R}^{M_{1,R}^1}$ associated with $M_{1,R}^1$ was already seen to be controlled by the spherical summation operator for the Hankel transform, so the associated maximal operator has the stated mapping property according to \cite{Colzani-Hankel}. The case $d=0$, cf. \eqref{eqn.MR0}, entails an analysis of three pieces, $M_{1,R}^{0,(1)}$, $M_{1,R}^{0,(2)}$, and $M_{1,R}^{0,(3)}$, corresponding to a suitable smooth partition of the interval $[1,R]$.

The piece $M_{1,R}^{0,(1)}(t,r)$ was seen to satisfy the estimate $\vert M_{1,R}^{0,(1)}(t,r)\vert\lesssim t^{-(\alpha+\frac{1}{2})}r^{-(\alpha+\frac{1}{2})}$, whence
\[\begin{split}
\vert S_{1,R}^{M_{1,R}^{0,(1)}}f(t)\vert &\lesssim \int_0^{R_0}\vert M_{1,R}^{0,(1)}(t,r)\vert \vert f(r)\vert\Delta(r)\,dr
\lesssim \frac{1}{t^{\alpha+\frac{1}{2}}}\int_0^{R_0}\frac{\vert f(r)\vert}{r^{\alpha+\frac{1}{2}}}\Delta(r)\,dr\\
&\lesssim\frac{1}{t^{\alpha+\frac{1}{2}}}\|f\|_{L^{p_0,1}([0,R_0],d\mu)} \cdot\|r\mapsto r^{-(\alpha+\frac{1}{2})}\|_{L^{p_1,\infty}([0,R_0],d\mu)}\\
&\lesssim \frac{1}{t^{\alpha+\frac{1}{2}}}\|f\|_{L^{p_0,1}(\R_+,d\mu)}.
\end{split}\]

The relevant level function for $S_{1,*}^{M_{1,R}^{0,(1)}}$ therefore satisfies the estimate
\[\begin{split}
d(\lambda)&=\mu\bigl(\bigl\{t\in[0,R_0]\,:\, \bigl\vert S_{1,*}^{M_{1,R}^{0,(1)}}f(t)\bigr\vert>\lambda\bigr\}\bigr)\leq \frac{1}{\lambda^{p_0}}\int_0^{R_0}\bigl\vert S_{1,*}^{M_{1,R}^{0,(1)}}f(t)\bigr\vert^{p_0}\Delta(t)\,dt\\
&\lesssim\frac{\|f\|_{L^{p_0,1}}^{p_0}}{\lambda^{p_0}} \int_0^{R_0}t^{2\alpha+1-(\alpha+\frac{1}{2})p_0}\,dt
\end{split}\]\texttt{}
where the integral is finite since $2\alpha+1-(\alpha+\frac{1}{2})p_0=2\alpha+1-(\alpha+\frac{1}{2})\frac{4\alpha+4}{2\alpha+3}=\frac{2\alpha+1}{2\alpha+3}>0>-1$ (since $\alpha>-\frac{1}{2}$ by standing assumption)
implying that $\|S_{1,*}^{M_{1,R}^{0,(1)}}f\|_{L^{p_0,\infty}(\R_+,d\mu)}\lesssim \|f\|_{L^{p_0,1}(\R_+,d\mu)}$ as claimed.

The mapping properties of the maximal operator associated with the piece $M_{1,R}^{0,(2)}$ are the same as for $M_{1,R}^{0,(1)}$ since $\vert M_{1,R}^{0,(2)}(t,r)\vert\lesssim t^{-(\alpha+\frac{1}{2})}r^{-(\alpha+\frac{1}{2})}$.

The most difficult piece,  $M_{1,R}^{0,(3)}$, gave rise to an operator that was controlled by the Carleson operator applied to the function $f\sqrt{\Delta}$, the upshot being the estimate \eqref{eqn.Carleson-estimate}. It is seen by close inspection of the argument on top of page 75 in \cite{Colzani-Hankel} the Carleson maximal operator is even bounded from $L^{p_0,1}$ into $L^{p_0,\infty}$ (the underlying measure space now being $\R_+$ with weighted Lebesgue measure $x^{2\alpha+1}dx$), so it follows that the maximal operator $S_{1,R}^{M_{1,R}^{0,(3)}}$ is bounded from $L^{p_0,1}(\R_+,d\mu)$ into $L^{p_0,\infty}(\R_+,d\mu)$. Recall here that $S_{1,R}^{M_{1,R}^{0,(3)}}f(t)$ is only considered for $0\leq t\leq R_0$, hence the stated result on the Carleson operator is applicable. Hence $S_{1,R}^{M_{1,R}^0}$ enjoys the stated endpoint mapping property.

The cases $d=2-M,\ldots,-2,-1$ now follow at once; above we have merely used that the relevant kernels were dominated by $ct^{-(\alpha+\frac{1}{2})}r^{-(\alpha+\frac{1}{2})}$. Since all the remaining operators satisfy the same estimates, we are effectively done; the maximal operator $S_{1,*}$ is bounded from $L^{p_0,1}(\R_+,d\mu)$ into $L^{p_0,\infty}(\R_+,d\mu)$.
\medskip

As for $S_{4,*}$ we will state the precise result as a lemma:

\begin{lem}\label{lemma.S4-endpoint}
The maximal operator $S_{4,*}$ is bounded from $L^{p_0,1}(\R_+,d\mu)$ into $L^2(\R,d\mu)$.
\end{lem}
\begin{proof}
Recall that the level function of a function $f\in L^{p_0}(\R_+,d\mu)$ is defined by $d_f(\lambda)=\mu(\{t\,:\, \vert f(t)\vert>\lambda\})$, hence
\[d_{S_{4,*}f}(\lambda)\leq\frac{1}{\lambda^2}\int_{R_0}^\infty\vert S_{4,*}f(t)\vert^2\Delta(t)\,dt\lesssim \frac{1}{\lambda^2}\int_{R_0}^\infty\Bigl(\frac{e^{-\rho t}}{t}\Bigr)^2\biggl|\int_0^{R_0/2}\frac{\vert f(r)\vert}{r^{\alpha+\frac{1}{2}}}\Delta(r)\,dr\biggr|^2\Delta(t)\,dt.\]
Observe that $\Delta(t)$ grows as $e^{\rho t}$ for $t\to\infty$, so that the $t$-integrand is dominated by $\frac{1}{t^2}$ on $[R_0,\infty)$. As for the inner integral, we intend to use the Lorentz space version of the H\"older inequality, that is (with $p_0=\frac{4\alpha+4}{2\alpha+3}$)
\[\biggl|\int_0^{R_0/2} f(r)\cdot r^{-(\alpha+\frac{1}{2})}\,d\mu(r)\biggr|\leq \|f\|_{L^{p_0,1}([0,R_0/2],d\mu)}\|r^{-(\alpha+\frac{1}{2})}\|_{L^{p_0',\infty}([0,R_0/2],d\mu)},\]
to which end it suffices to show that $g:[0,R_0/2]\to\R$, $r\mapsto r^{-(\alpha+\frac{1}{2})}$ belongs to $L^{p_0',\infty}([0,R_0/2],d\mu)$. This is easy: It follows from the estimate
\[\begin{split}
d_g(\gamma)&=\mu(\{r\in[0,R_0/2]\,:\, r^{-(\alpha+\frac{1}{2})}>\gamma\}) \leq \mu(\{r\geq 0\,:\,r^{\alpha+\frac{1}{2}}<\tfrac{1}{\gamma}\})\\
&=\mu\Bigl( r\in\R_+\,:\, r<\gamma^{-\frac{1}{\alpha+\frac{1}{2}}}\Bigr)\leq \Bigl(\gamma^{-\frac{1}{\alpha+\frac{1}{2}}}\Bigr)^{2\alpha+2}=\gamma^{-\frac{2\alpha+2}{\alpha+\frac{1}{2}}},
\end{split}\]
that $d_g(\gamma)^{1/p_0'}\leq \bigl(\gamma^{-\frac{2\alpha+2}{\alpha+\frac{1}{2}}}\bigr)^{\frac{2\alpha+1}{4\alpha+4}}=\gamma^{-1}$, and therefore
\[\|g\|_{L^{p_0',\infty}([0,R_0/2],d\mu)}=\sup_{\gamma>0}\gamma d_g(\gamma)^{1/p_0'}\leq 1.\]
\end{proof}
This completes the proof of Theorem \ref{thm.endpoint-Lorentz}.

\begin{rem}
Lemma \ref{lemma.S4-endpoint} should be seen as a ``non-Euclidean'' analogue of the result from \cite{RomeraSoria} and the statement is new even for rank one symmetric spaces.
\end{rem}


\begin{thebibliography}{25}
\normalsize
\baselineskip=17pt


\bibitem{Anker-Damek-Yacoub}
J.P. Anker and C.~Damek, E.~Yacoub.
\newblock Spherical analysis on harmonic $an$ groups.
\newblock {\em Ann. Scuola Norm. Sup. Pisa Cl. Sci. (4)}, 23(4):643--679, 1996.

\bibitem{Brandolini-Gigante}
L.~Brandolini and G.~Gigante.
\newblock Equiconvergence theorems for {C}h\'ebli-{T}rim\`eche hypergroups.
\newblock {\em Ann. Sc. Norm. Super. Pisa Cl. Sci. (5)}, 8(2):211--265, 2009.

\bibitem{Clerc-Stein}
J.~L. Clerc and E.~M. Stein.
\newblock {$L^{p}$}-multipliers for noncompact symmetric spaces.
\newblock {\em Proc. Nat. Acad. Sci. U.S.A.}, 71:3911--3912, 1974.

\bibitem{Colzani-Hankel}
L.~Colzani.
\newblock Hankel transform on {L}orentz spaces.
\newblock {\em Colloq. Math.}, 60/61(1):71--76, 1990.

\bibitem{ElKamel-Yacoub}
J.~El~Kamel and Ch. Yacoub.
\newblock Almost everywhere convergence of inverse {D}unkl transform on the
  real line.
\newblock {\em Tamsui Oxf. J. Math. Sci.}, 25(3):259--267, 2009.

\bibitem{Erdelyi-asymptotic}
A.~Erd{\'e}lyi.
\newblock {\em Asymptotic expansions}.
\newblock Dover Publications, Inc., New York, 1956.

\bibitem{Fefferman-ballmultiplier}
C.~Fefferman.
\newblock The multiplier problem for the ball.
\newblock {\em Ann. of Math. (2)}, 94:330--336, 1971.

\bibitem{FJ}
M.~Flensted-Jensen.
\newblock {P}aley--{W}iener type theorems for a differential operator connected
  with symmetric spaces.
\newblock {\em Ark. Mat.}, 10:143--162, 1972.

\bibitem{Koornwinder-FJ}
M.~Flensted-Jensen and T.~Koornwinder.
\newblock The convolution structure for {J}acobi function expansions.
\newblock {\em Ark. Mat.}, 11:245--262, 1973.

\bibitem{Grafakos}
L.~Grafakos.
\newblock {\em Modern {F}ourier analysis}, volume 250 of {\em Graduate Texts in
  Mathematics}.
\newblock Springer, New York, second edition, 2009.

\bibitem{Herz}
C.~Herz.
\newblock On the mean inversion of {F}ourier and {H}ankel transforms.
\newblock {\em Proc. Nat. Acad. Sci. U.S.A.}, 40:996--999, 1954.

\bibitem{Johansen-exp1}
T.~R. Johansen.
\newblock {$L^p$}-results for fractional integration and multipliers for the
  {J}acobi transform.
\newblock {\em preprint, available at
  http://johansen.math.uni-kiel.de/forschung/Publikationen}.

\bibitem{Johansen-nonint}
T.~R. Johansen.
\newblock On a class of non-integrable multipliers for the {J}acobi transform.
\newblock {\em preprint, available at
 http://johansen.math.uni-kiel.de/forschung/Publikationen}, submitted
  (2010).

\bibitem{Kanjin}
Y.~Kanjin.
\newblock Convergence almost everywhere of {B}ochner--{R}iesz means.
\newblock {\em Ann. Sci. Kanazawa Univ.}, 25:11--15, 1988.

\bibitem{Kanjin-summation}
Y.~Kanjin.
\newblock Convergence and divergence almost everywhere of spherical means for
  radial functions.
\newblock {\em Proc. Amer. Math. Soc.}, 103(4):1063--1069, 1988.

\bibitem{Koornwinder-newproof}
T.~Koornwinder.
\newblock A new proof of a {P}aley-{W}iener type theorem for the {J}acobi
  transform.
\newblock {\em Ark. Mat.}, 13:145--159, 1975.

\bibitem{Meaney-divergent}
C.~Meaney.
\newblock Divergent {J}acobi polynomial series.
\newblock {\em Proc. Amer. Math. Soc.}, 87(3):459--462, 1983.

\bibitem{Meaney-Prestini_invtrans}
C.~Meaney and E.~Prestini.
\newblock Almost everywhere convergence of inverse spherical transforms on
  noncompact symmetric spaces.
\newblock {\em J. Funct. Anal.}, 149(2):277--304, 1997.

\bibitem{Meaney-Prestini}
C.~Meaney and E.~Prestini.
\newblock Bochner--{R}iesz means on symmetric spaces.
\newblock {\em Tohoku Math. J. (2)}, 50(4):557--570, 1998.

\bibitem{Paris-Kaminski}
R.~B. Paris and D.~Kaminski.
\newblock {\em Asymptotics and {M}ellin--Barnes integrals}.
\newblock Number~85 in Encyclopedia of mathematics and its applications.
  Cambridge University Press, Cambridge, United Kingdom, 2001.

\bibitem{Prestini-summation}
E.~Prestini.
\newblock Almost everywhere convergence of the spherical partial sums for
  radial functions.
\newblock {\em Monatsh. Math.}, 105(3):207--216, 1988.

\bibitem{Prestini-weighted}
E.~Prestini.
\newblock Weights of exponential type.
\newblock {\em Monatsh. Math.}, 127(4):337--341, 1999.

\bibitem{RomeraSoria}
E.~Romera and F.~Soria.
\newblock Endpoint estimates for the maximal operator associated to spherical
  partial sums on radial functions.
\newblock {\em Proc. Amer. Math. Soc.}, 111(4):1015--1022, 1991.

\bibitem{Schindler}
S.~Schindler.
\newblock Some transplantation theorems for the generalized {M}ehler transform
  and related asymptotic expansions.
\newblock {\em Trans. Amer. Math. Soc.}, 155:257--291, 1971.

\bibitem{Stanton-Tomas}
R.~J. Stanton and P.~A. Tomas.
\newblock Expansions for spherical functions on noncompact symmetric spaces.
\newblock {\em Acta Math.}, 140(3-4):251--276, 1978.

\bibitem{Watson}
G.~N. Watson.
\newblock {\em A treatise on the theory of {B}essel functions}.
\newblock Cambridge University Press, Cambridge, second edition, 1944.

\end{thebibliography}
\end{document}